\documentclass{article}
\usepackage{amssymb,amsmath,mathrsfs}
\usepackage[colorlinks=false]{hyperref}
\usepackage{comment}
\usepackage{diagbox}
\usepackage{empheq}
\usepackage{geometry}
\usepackage{diagbox}
\usepackage{manyfoot}
\usepackage{graphicx,color}
\usepackage[outercaption]{sidecap}
\usepackage{xfrac}
\usepackage{subfig}
\usepackage{tikz}
\usepackage{enumitem}
\usepackage{pdfsync}
\usepackage{fancyhdr}
\usepackage[nottoc,notlof,notlot]{tocbibind}
\usepackage[titles,subfigure]{tocloft}

\usepackage{amssymb}
\usepackage{amsmath}
\usepackage{latexsym}
\usepackage{amsthm}
\usepackage{eucal}
\usepackage{amsthm}
\theoremstyle{plain}
\newtheorem{theorem}{Theorem}[section]
\newtheorem{corollary}[theorem]{Corollary}

\newtheorem{proposition}[theorem]{Proposition}
\theoremstyle{definition}
\newtheorem{definition}[theorem]{Definition}
\newtheorem{example}[theorem]{Example}
\newtheorem{remark}[theorem]{Remark}

\definecolor{MyDarkGreen}{rgb}{0,0.50,0.04}

\def\e{\varepsilon}
\def\rr{{\mathbb R}}

\def\NN{{\mathbb N}}

\DeclareMathOperator*{\dist}{dist}
\DeclareMathOperator*{\spann}{span}
\DeclareMathOperator*{\supp}{supp}

\DeclareMathOperator*{\reg}{reg}
\DeclareMathOperator*{\diam}{diam}

\def\XXint#1#2#3{{\setbox0=\hbox{$#1{#2#3}{\int}$} 
  \vcenter{\hbox{$#2#3$}}\kern-.5\wd0}}

\numberwithin{equation}{section}

\title{Nonlocal-to-local variational limits of energies on partitions}
\author{
{\sc Roberto Alicandro}
\\ \small Dipartimento di Matematica ``R. Caccioppoli'',
 Universit\`a di Napoli Federico II,\\
 \small
via Cintia, 80126 Napoli, Italy\\
\small roberto.alicandro@unina.it
\\
{\sc Giuliana Fusco}
\\ \small Scuola Superiore Meridionale,\\
 \small
via Mezzocannone 4, 80134 Napoli, Italy\\
\small g.fusco@ssmeridionale.it
\\
{\sc Chiara Leone}
\\ \small Dipartimento di Matematica ``R. Caccioppoli'',
 Universit\`a di Napoli Federico II,\\
 \small
via Cintia, 80126 Napoli, Italy\\
\small chileone@unina.it
}
\date{}

\begin{document}
\maketitle

\noindent
{\bf Abstract.}
We study, through a $\Gamma$-convergence procedure, the asymptotic behaviour of a general class of nonlocal energies defined on functions taking values in a finite set
$Y=\{y_1,\dots,y_N\}\subset\mathbb R^m$. These energies are written in terms of double integrals  accounting for pairwise interactions in an open set $\Omega$, weighted by a kernel $K_\varepsilon$.  If the latter is controlled from above and below by suitable non-negative radial kernels with unit mass concentrating at the origin, then, up to subsequences, the $\Gamma$-limit has domain $BV(\Omega;Y)$ and is given by a local surface energy. We show that the locality of the $\Gamma$-limit is not  guaranteed if the concentration of mass fails. For a class of reference energies, we obtain a more explicit
representation of the limit energy, also for non-radial kernels, recovering
the limits of isotropic and anisotropic fractional perimeters.
We also prove a homogenization result for periodic kernels.
\bigskip

\smallskip


\section{Introduction}

Over the last two decades, increasing attention has been devoted to variational models involving nonlocal energies, which depend on interactions between points at a finite distance, providing a natural
framework for phenomena that cannot be adequately described by purely local energies. Such phenomena arise, for instance, in the description of phase transitions with long-range interactions \cite{AB, CH, H, SV}, in nonlocal minimal surfaces \cite{CRS}, in image processing \cite{BN, GO}, and in continuum mechanics through peridynamic theory  \cite{BMP}.

A prototypical example of such energies is given by functionals of the form
\begin{equation*}
\int_{\Omega}\int_{\Omega} f(x,y,u(x),u(y))\, dx\, dy
\end{equation*}
where $\Omega$ is an open set of $\rr^d$ and $f$ accounts for the interaction between pairs of points and depends
on their relative position and on the states of the system at both points.
\begin{comment}
   and the function $f$ describes the cost of the
interaction between the states $u(x)$ and $u(y)$. 
\end{comment}

A central question is whether, and in which form, a local description is recovered when the range
of the interactions becomes infinitesimal. The problem was first investigated by  Bourgain, Brezis, and Mironescu in their celebrated paper \cite{BBM}, where they studied the asymptotic behaviour of the Gagliardo seminorm $\displaystyle [u]_{W^{s,p}(\Omega)}$, as $s\to 1^-$, showing  that
$$
(1-s)[u]_{W^{s,p}(\Omega)}=(1-s)\int_\Omega\int_\Omega \frac{\ |u(y)-u(x)|^p\  }{\ |y-x|^{d+sp}}\, dx\, dy 
$$ 
approximates, up to a multiplicative constant,  the $p$-th power of the $L^p$ norm of $\nabla u$ for $p>1$ and $u\in W^{1,p}(\Omega)$, and the total variation $|Du|(\Omega)$ for $p=1$
and $u\in BV(\Omega)$ (see also \cite{D02}). The result was subsequently extended in \cite{P2} in terms of $\Gamma$-convergence.

\begin{comment}
    The origin of such nonlocal analysis goes back to the celebrated paper by Bourgain, Brezis and Mironescu \cite{BBM},
where the authors considered, for a bounded smooth open set
$\Omega\subset\mathbb{R}^d$ and a family of radial kernels
$\rho_\varepsilon\in L^1(\mathbb{R}^d)$ of unit mass concentrating at the origin,
energies of the form
$$
\int_{\Omega}\int_\Omega \rho_\varepsilon(y-x)\frac{|u(y)-u(x)|^p}{|y-x|^{p}}\,dx\,dy ,
\qquad u\in L^p(\Omega).
$$
They proved that these energies converge, up to a dimensional constant, to the local energy
$$\int_\Omega|\nabla u|^p\,dx,$$ providing a new characterization of
$W^{1,p}(\Omega)$ for $p>1$ and of $BV(\Omega)$ for $p=1$. In the latter case, the
convergence to a multiple of the total variation $|Du|(\Omega)$ for every
$u\in BV(\Omega)$, left open in \cite{BBM} for $d\ge2$, was later proved by
D\'avila \cite{D02}. Ponce extended these results to non-radial kernels, carrying out the corresponding $\Gamma$-convergence
analysis \cite{P2}, and proving the compactness of sequences with equibounded energies \cite{P}. Both
results play a key role in our analysis. 
\end{comment}

For energies with superlinear growth, a general asymptotic analysis  as $\e\to 0$ of families of functionals of the form
\begin{equation}\label{doubleinteps}
\int_\Omega\int_\Omega f_\e(x, y, u(y)-u(x))\, dx\, dy,
\end{equation}
under concentration of the energies on $x=y$, has been recently provided in \cite{AABPT},  leading to a general class of energies whose $\Gamma$-limits are local integral functionals on Sobolev space of the form
$$
\int_{\Omega} f(x,\nabla u)\, dx.
$$
This result
has been applied, among others, to periodic and stochastic homogenization, to
perforated domains (see for instance \cite{AGL, BF, DFZ, SZZ}), and, very recently, to the derivation of
lower-dimensional models for nonlocal thin films \cite{AT26}. The same approach has also been extended to
energies leading to free discontinuity functionals \cite{BDSTV}.

\begin{comment}
    For energies with superlinear growth, a general theory of convolution-type
functionals has been recently developed by the first author, Ansini, Braides, Piatnitski,
and Tribuzio \cite{AABPT}. In this book it has been proved that under growth and decay conditions on the kernels, every
$\Gamma$-limit is a local integral functional on Sobolev space. This result
has been applied, among others, to periodic and stochastic homogenization, to
perforated domains (see for instance \cite{AGL, BF, DFZ, SZZ}), and, very recently, to the derivation of
lower-dimensional models for nonlocal thin films \cite{AT26}. The same approach has also been extended to
energies leading to free discontinuity functionals \cite{BDSTV}.\\
\end{comment}

In this paper we consider the case in which the order parameter $u$ takes only
finitely many values; i.e. $u\colon\Omega\to Y=\{y_1,\dots,y_N\}\subset\mathbb{R}^m$, so
that every admissible configuration is a partition of $\Omega$ into the phases
$\{u=y_i\}$, for every $i=1,\dots,N$. Such a description is the natural setting whenever the system can only be represented by a finite number of states. 
Typical examples are the components of a mixture of immiscible fluids, as in the
vector-valued Cahn--Hilliard theory \cite{B90}, the admissible orientations of a small
droplet of a liquid crystal \cite{V}, and the spin values of a magnetic system of
Ising type \cite{ABCS, AG}.
The local theory for this kind of phenomena typically involves interfacial energies of the form 
\begin{equation}
\label{intro_1}
\int_{S_u}\varphi\big(x,u^+,u^-,\nu_u\big)\,d\mathcal{H}^{d-1},
\end{equation}
 where
$S_u$ is the jump set of $u$, $\nu_u$ is the
measure-theoretic normal to $S_u$, and $u^\pm$ are the traces of $u$ on the two
sides of $S_u$. A general variational analysis, studying  integral representation, lower semicontinuity,
relaxation, and homogenization of this kind of energies, was developed by Ambrosio and Braides
\cite{AB1,AB2}.

The question we address here is to understand how such interfacial
energies emerge from nonlocal models. In the discrete setting, this problem was investigated in \cite{AG} for Ising-type energies of spin systems on a lattice. In the continuum setting, this kind of nonlocal-to-local analysis  was developed in \cite{ADM},
where the asymptotic behaviour of the $s$-fractional perimeters as $s\to 1^-$ was studied. In particular, it was shown that $\displaystyle (1-s)[\chi_E]_{W^{s,1}(\Omega)}$ $\Gamma$-converges to
$\omega_{d-1}P(E,\Omega)$, where $P(\cdot,\Omega)$ denotes the classical  De Giorgi perimeter. More recently, in \cite{ACFS25} the authors considered energies on partitions
depending on the Riesz $s$-fractional gradient, for real-valued functions taking
finitely many values, and proved that, as $s\to1^-$, the $\Gamma$-limit is a
local interfacial energy as in \eqref{intro_1}.

In this paper, we consider a general class of functionals of the form \eqref{doubleinteps}
defined on  $L^1(\Omega;Y)$, where $f_\varepsilon\colon\Omega\times\Omega\times Y\times Y\to[0,+\infty]$
is a Borel function such that $f_\varepsilon(x,y,u,u)=0$. 
The last condition encodes the fact that uniform configurations minimize the energy. Setting
\begin{equation*}
    K_\varepsilon(x,y,u,v):=
    \begin{cases}
        0 &\text{if}\ u=v\\
        f_\varepsilon(x,y,u,v)\frac{|y-x|}{|v-u|}  &\text{if}\ u \neq v,
    \end{cases}
\end{equation*}
the integrand $f_\varepsilon$ can be written as
\begin{equation*}
    f_\e(x,y,u,v)=K_\varepsilon(x,y,u,v)\frac{|v-u|}{|y-x|}
\end{equation*}
and, in turn,
\begin{equation*}
{F}_\varepsilon(u)=\int_{\Omega}\int_{\Omega}K_\e(x,y,u(x),u(y)) \frac{|u(y)-u(x)|}{|y-x|}\, dx\, dy.
\end{equation*}
We assume that the kernels $K_\varepsilon$ are comparable with a family of radial kernels $\rho_\varepsilon$ with unit mass 
concentrating at the origin, namely
$$
C_1\rho_\varepsilon(y-x)\le K_\varepsilon(x,y,u,v)\le C_2\rho_\varepsilon(y-x)
\qquad\text{for every } u,v\in Y,\ u\neq v
$$
(see Section \ref{Setting of the problem} for the precise assumptions). Our main result, Theorem \ref{main thm}, establishes compactness and integral representation for this
class of energies. Namely, we prove that, up to
subsequences, the functionals $F_\varepsilon$ $\Gamma$-converge, as $\e\to 0^+$, to a functional of the form
$$
F(u)=\begin{cases}
\displaystyle\int_{S_u}\varphi\big(x,u^+,u^-,\nu_u\big)\,d\mathcal H^{d-1}
& \text{if } u\in BV(\Omega;Y),\\[2mm]
+\infty & \text{otherwise}.
\end{cases}
$$
 When $K_\varepsilon(x,y,u,v)=\rho_\varepsilon(y-x)$, the $\Gamma$-limit has a more explicit representation in terms of a measure $\mu$ associated to the kernels $\rho_\e$ (see Theorem \ref{Thm_gamma_limite_esplicito}). 
Our analysis covers both kernels of fractional
type, isotropic and anisotropic,
thus including the fractional perimeters considered in \cite{ADM} and their anisotropic versions
studied in \cite{L14,F25} (see Examples \ref{ex_frac_per}--\ref{ex:anisotropic}). As a consequence of our $\Gamma$-convergence results, in
Section \ref{Convergence of minimum problems.} we prove the convergence of minimum
problems with Dirichlet boundary conditions and with prescribed volume fractions of
the phases (Theorems \ref{convergence of minimum problem} and \ref{teo_volume_fratto}).

The proof of Theorem \ref{main thm}, carried out in Section \ref{Proof main thm},
relies on an abstract method originally exploited in the
homogenization theory for multiple integrals  \cite{DM,BD}. It amounts to applying a localization argument which allows us to regard our functionals and their $\Gamma$-limits as defined on functions and sets and then to prove that all the assumptions of an integral representation result for functionals defined on partitions proved in \cite{BFLM}  are fulfilled (see
Theorem \ref{representation thm}). A key point is to show that the $\Gamma$-limit is the restriction of a measure,
and in particular that it is subadditive (see Proposition \ref{Prop subadditivity}). In this respect, the hypothesis of mass concentration at the origin \eqref{eq:nuclei3} plays a crucial role in showing that the nonlocality of the approximating functionals disappears in the limit.

 An application of our  result is the homogenization of nonlocal energies
on partitions, which we address in Section \ref{Homogenization}. We consider kernels
of the form $$K_\varepsilon(x,y,u,v)=\varepsilon^{-d}K(x/\varepsilon,y/\varepsilon,u,v),$$
with $K$ $1$-periodic with respect to simultaneous translations of the two spatial variables (see \eqref{periodicita}), so that the heterogeneity of the medium and the range of the interactions are of the same order $\varepsilon$. 
We prove that the family of functionals $F_\varepsilon$ $\Gamma$-converges, as $\e\to 0^+$, to a homogeneous
interfacial energy, whose density is defined by an asymptotic nonlocal cell formula
(see Theorem \ref{teo:homogenization}).

Finally, in Section \ref{Approximation of non local functional} we show that the
 condition \eqref{eq:nuclei3} is essential to recover the locality of the $\Gamma$-limit. Indeed we exhibit two examples where \eqref{eq:nuclei3} fails and the $\Gamma$-limit is the sum of an interfacial energy and of a
genuinely nonlocal term.

\section{Notation and Preliminaries}

In what follows, given $x,y \in \rr^d$ we denote by $\langle x,y\rangle$ the standard scalar product in $\rr^d$ and we set $|x|=\sqrt{\langle x,x\rangle}$. Given $k\in\NN$, we denote by $\mathcal{H}^{k}$ the $k$-dimensional Hausdorff measure, by $\sigma_{k}$ the $k$-dimensional Hausdorff measure of the unit sphere $S^{k}:=\{x\in \rr^{k+1}:|x|=1\}$ in $\rr^{k+1}$, and by $\omega_{k}$ the volume of the unit ball in $\rr^{k}$.
We denote by $\mathcal{A}(\Omega)$ the class of open subsets in $\Omega$ and by
$\mathcal{A}^{\reg}(\Omega)$ the family of open subsets of $\Omega$ with Lipschitz boundary. 
For any $\nu \in S^{d-1}$, fix an orthonormal basis $\{\nu^1,\dots,\nu^{d-1},\nu\}$  of $\rr^d$, and let $Q_\nu:=\{x \in \rr^d: |\langle x, \nu^i\rangle |<1/2,\ \text{for}\ i=1, \dots,d-1,\  |\langle x,\nu\rangle|<1/2\}$ be the unit cube centered at the origin with two of its faces orthogonal to $\nu$. Given $r>0$ and $x_0\in \rr^d$, we set $Q_\nu(x_0,r):=x_0+rQ_\nu$, the cube centered at $x_0$, with side length $r$, and oriented in the direction $\nu$. Unless otherwise stated, $C$ will always denote a generic strictly positive constant that may change from line to line.

\begin{comment}
\begin{definition}
\label{classe densa}
     Given $\mathcal{A}(\Omega)$, a subclass $\mathcal{A}^{\mathcal{D}}(\Omega)\subset \mathcal{A}(\Omega)$ is said to be dense (or rich) in $\mathcal{A}(\Omega)$ if for every $A', A \in \mathcal{A}(\Omega)$ with $A'\subset \subset A$, there exists $A''\in \mathcal{A}^{\mathcal{D}}(\Omega)$ such that $A'\subset \subset A''\subset \subset A$.
\end{definition}
\end{comment}

\noindent For the reader's convenience, we recall the definition of $BV(\Omega;Y)$ and the compactness result related to this space, which can be deduced by the corresponding properties of sets of finite perimeter (see \cite{AB1,AFP,G}).

\begin{definition}
    Let $Y\subset \rr^{m}$ be a finite set whose elements are denoted by $y_1,\dots, y_N$. We denote by $BV(\Omega;Y)$ the class of Borel functions $u:\Omega \rightarrow Y$ whose level sets $\{u=y_i\}$ are sets of finite perimeter in $\Omega$. We denote by $S_u$ the Borel set
    \begin{equation*}
        S_u=\bigcup_{y_i \in Y}\partial ^{*}\{u=y_i\}=\bigcup_{y_i,y_j\in Y, y_i\neq y_j}\partial^{*}\{u=y_i\}\cap\partial^{*}\{u=y_j\},
    \end{equation*}
where, for a set $E\subset\rr^d$, of finite perimeter we denote by $\partial^{*} E$ the essential boundary of $E$.
    
Moreover, for $\mathcal{H}^{d-1}$-almost every $x \in S_u$ there exists a unit vector $\nu_u(x) \in S^{d-1}$ such that 
\begin{equation*}
    \lim_{r\rightarrow 0^+}r^{-d}|\{z \in B_{r}(x):\langle z-x,\nu_u(x)\rangle>0, u(z)\neq y_i\}|=0
\end{equation*}
and
\begin{equation*}
    \lim_{r\rightarrow 0^+}r^{-d}|\{z \in B_{r}(x):\langle z-x,\nu_u(x)\rangle<0, u(z)\neq y_j\}|=0.
\end{equation*}
We set $y_i=u^+(x)$ and $y_j=u^-(x)$. The triplet $(u^+,u^-,\nu_u)$ is defined $\mathcal{H}^{d-1}$-almost everywhere on $S_u$, and it is uniquely determined, up to a change of sign of $\nu_u$ and an interchange of $u^+$, $u^-$.
\end{definition}

\begin{theorem}[Compactness]
\label{teo_compactness}
    Let $u_k \in BV(\Omega;Y)$ such that
    \begin{equation*}
        \sup_k \mathcal{H}^{d-1}(S_{u_k})<+\infty.
    \end{equation*}
    Then there exists a subsequence (not relabeled) and $u \in BV(\Omega;Y)$ such that $u_k\rightarrow u$ in $L^1$.
\end{theorem}

\noindent We recall a characterization of measures as set functions due to De Giorgi and Letta (1977) {(see \cite[Theorem 10.2]{BD})}.
\begin{theorem}
\label{De Giorgi Letta Criterion}
    Let $\alpha : \mathcal{A}(\Omega)\rightarrow [0,+\infty]$ be an increasing set function. The following statements are equivalent:
    \begin{itemize}
        \item[(i)] $\alpha$ is the restriction to $\mathcal{A}(\Omega)$ of a Borel measure on $\Omega$;
        \item[(ii)] $\alpha$ is subadditive, superadditive and inner regular;
        \item[(iii)] the set function
        \begin{equation*}
            \beta(E)=\inf\{\alpha(U): U\in \mathcal{A}(\Omega),E\subseteq U\}
        \end{equation*}
        defines a Borel measure on $\Omega$.
    \end{itemize}
\end{theorem}

\noindent The following integral representation result for functionals defined on partitions has been proved in \cite[Theorem 3]{BFLM}.

\begin{theorem}
\label{representation thm}
    Let $F:BV(\Omega;Y)\times \mathcal{A}(\Omega) \rightarrow [0,+\infty)$ satisfy for every $(u,A)\in BV(\Omega;Y)\times \mathcal{A}(\Omega)$ the following hypotheses:
    \begin{itemize}
        \item[(i)] $F(u,\cdot)$ is the restriction to $\mathcal{A}(\Omega)$ of a Radon measure;
        \item[(ii)] $F(u,A)=F(v,A)$ whenever $u=v$ a.e. on $A\in \mathcal{A}(\Omega)$;
        \item[(iii)] $F(\cdot,A)$ is $L^{1}(\Omega;\rr^m)$ lower semicontinuous;
        \item[(iv)] there exist $C_1,C_2>0$ such that 
        \begin{equation*}
            C_{1}\mathcal{H}^{d-1}(S_u\cap A)\leq F(u,A)\leq C_{2}\mathcal{H}^{d-1}(S_u\cap A).
        \end{equation*}
    \end{itemize}
    Then for every $u \in BV(\Omega;Y)$ and $A \in \mathcal{A}(\Omega)$
    \begin{equation*}
        F(u,A)=\int_{S_{u}\cap A}\varphi(x,u^{+}, u^{-},\nu_{u})\, d\mathcal{H}^{d-1},
    \end{equation*}
    with
    \begin{equation}
    \label{densita}\varphi(x_{0},a,b,\nu)=\limsup_{r\rightarrow 0^+}\frac{m(u_{x_0,a,b,\nu},Q_{\nu}(x_{0},r))}{r^{d-1}}
    \end{equation}
    for all $x_{0} \in \Omega, a, b \in Y, \nu \in S^{d-1}$, where 
    
    \begin{equation}
    \label{eq: def u}
        u_{x_0,a, b, \nu}(x):=
        \begin{cases}
            a\ &\text{if} \ \langle x-x_0,\nu\rangle >0\\
            b\ &\text{if} \ \langle x-x_0,\nu\rangle\leq0
        \end{cases}
    \end{equation}
    and for any  $(v,A)\in BV(\Omega;Y)\times \mathcal{A}(\Omega)$,
    \begin{equation*}
        m(v,A):=\inf\{F(w,A):w\in BV(\Omega;Y), w=v\ \textit{in a neighborhood of}\  \partial A\}.
    \end{equation*}
\end{theorem}

\noindent In the sequel, we will use the notation $u_{a,b,\nu}$ to denote the function 
defined in \eqref{eq: def u} with $x_0=0$.

\section{Setting of the problem and main result}\label{Setting of the problem}

Let $\Omega \subset\rr^d$ be a bounded open set with Lipschitz boundary and $Y\subset \rr^m$ be a finite set endowed with the discrete topology. Given $\varepsilon>0$,  let $K_\varepsilon:\Omega\times \Omega\times Y\times Y \rightarrow [0,+\infty ]$ be a Borel function
and let
$F_\e:L^1(\Omega;Y)\to[0,+\infty]$ be defined by

\begin{equation}\label{rescale_energy}
{F}_\varepsilon(u)=\int_{\Omega}\int_{\Omega}K_\e(x,y,u(x),u(y)) \frac{|u(y)-u(x)|}{|y-x|}\, dx\, dy.
\end{equation}

\noindent We fix the following set of hypotheses on $K_{\varepsilon}(x,y,u,v)$: there exists a  family of Borel functions $\rho_\varepsilon:\rr^d\to [0,+\infty)$ satisfying
   
\begin{empheq}[left=\empheqlbrace]{alignat=2}
&\rho_{\varepsilon}(\cdot)\  \text{is radial}
&\qquad&\forall \varepsilon>0,
\label{eq:nuclei1}\\
&
\int_{\rr^{d}}\rho_{\varepsilon}(z)\,dz
 = 1
&& \forall \varepsilon>0,
\label{eq:nuclei2}\\
&\lim_{\varepsilon\to0}
\int_{|z|>\eta}\rho_{\varepsilon}(z)\,dz
=0
&& \forall \eta>0,
\label{eq:nuclei3}
\end{empheq}
   such that for a.e. $x,y\in\Omega$  and for every $(u,v)\in Y\times Y$  with $u\neq v$,
it holds:
\begin{itemize}
    \item[(H1)]\textit {(Coerciveness)} $\exists \,C_1>0$ such that
    $$ K_{\varepsilon}(x,y,u,v)\geq C_1\rho_{\e}(y-x), $$
 \item[(H2)] \textit{(Growth)} $ \exists\, C_2>0$ such that $$K_{\varepsilon}(x,y,u,v)\leq C_2 \rho_{\e}(y-x).$$   
\end{itemize}

\noindent  We also define a localized version of our energies. For any $A \in \mathcal{A}(\Omega)$ and for $ u \in L^1(A;Y)$ we set
\begin{equation*}  
        F_{\varepsilon}(u,A):=
    \displaystyle 
    \int_{A}\int_{A}K_{\varepsilon}(x,y,u(x),u(y))\frac{|u(y)-u(x)|}{|y-x|}\, dx\, dy.
\end{equation*}
Notice that, when $A=\Omega$, the localized functional agrees with
the functional defined in \eqref{rescale_energy}, namely
$F_\varepsilon(u,\Omega)=F_\varepsilon(u)$.
We will denote by $F'(\cdot,A)$ and $F''(\cdot,A)$ the $\Gamma$-$\displaystyle \liminf_{\varepsilon\rightarrow 0}$ and $\Gamma$-$\displaystyle \limsup_{\varepsilon\rightarrow 0}$ of $F_{\varepsilon}(\cdot,A)$  with respect to the $L^{1}(A;\rr^m)$ strong topology, respectively.

\begin{comment}
\begin{remark}
    For the sake of simplicity we prefer to deal with the class $\mathcal{A}^{\reg}(\Omega)$, extending suitably $F'(\cdot,A), F''(\cdot,A)$ to $\mathcal{A}(\Omega)$ as last step.
\end{remark}
\end{comment}

 We now state the main result of the paper.
\begin{theorem}
\label{main thm}
    Let $F_{\varepsilon}$ be defined by \eqref{rescale_energy}, with $K_{\varepsilon}(x,y,u,v)$ satisfying {\rm (H1)}-{\rm (H2)}. Then for any sequence $\varepsilon_j\rightarrow0$ there exist a subsequence $\varepsilon_{j_k}$ and a function $\varphi:\Omega \times Y \times Y\times  S^{d-1}\rightarrow [0,+\infty)$ such that the functionals $F_{\varepsilon_{j_k}}$ $\Gamma$-converge with respect to the $L^{1}(\Omega;\rr^m)$ strong topology to the functional $F:L^{1}(\Omega;Y)\rightarrow [0,+\infty]$ defined by 
    \begin{equation}
        F(u):=
        \begin{cases}
            \displaystyle
            \int_{S_{u}}\varphi(x, u^{+}, u^{-},\nu_{u})\, d\mathcal{H}^{d-1}\ &\text{if}\ u\in BV(\Omega;Y)\\
            +\infty \ &\text{otherwise}.
        \end{cases}
        \label{gamma limite}
    \end{equation}
    
    \noindent Moreover, for every $u \in BV(\Omega;Y)$ and $A \in \mathcal{A}^{\reg}(\Omega)$, it holds
    \begin{equation*}
        \Gamma(L^1)\,\text{-}\lim_{k} F_{\varepsilon_{j_k}}(u,A)=F(u,A):=\int_{S_{u}\cap A}\varphi(x,u^{+},u^{-},\nu_{u})\, d\mathcal{H}^{d-1}.
    \end{equation*}
\end{theorem}

\begin{remark}
\label{rmk_truncated}
    Given $R>0$ set 
    \begin{equation*}
    F^R_{\e}(u,A)=\int_{\{(x,y)\in A \times A:|y-x|\leq R\}
    }K_\e(x,y,u(x),u(y)) \frac{|u(y)-u(x)|}{|y-x|}\, dx\, dy.
\end{equation*}
We note that for every $u \in L^1(A;Y)$, by {\rm (H2)} it holds
    \begin{equation*}
    \begin{split}
         \int_{{\{(x,y)\in A\times A:|y-x|> R\}}
    }&K_\e(x,y,u(x),u(y)) \frac{|u(y)-u(x)|}{|y-x|}\, dx\, dy\\
    &\leq C_2\int_{{\{(x,y)\in A \times A:|y-x|> R\}}
    }\rho_\e(y-x) \frac{|u(y)-u(x)|}{|y-x|}\, dx\, dy\\
    &\leq \frac{C}{R}|A|\int_{|z|>R}\rho_\e(z)\, dz.
    \end{split}
    \end{equation*}
    Hence, by \eqref{eq:nuclei3},
    \begin{equation*}
        0\leq \lim_{\e\to 0}\sup_{u \in L^{1}(A;Y)}(F_\e(u,A)-F_\e^{R}(u,A))\leq \lim_{\e \to 0}\frac{C}{R}|A|\int_{|z|>R}\rho_\e(z)\, dz=0.
    \end{equation*}
We deduce that $F_\e$ and $F_\e^R$ share the same asymptotic behavior. In particular, Theorem \ref{main thm} still holds if we replace $F_\e$ by $F_\e^R$.
\end{remark}

\noindent We postpone the proof of Theorem \ref{main thm}, as it will be a consequence of some propositions which show that our limit functionals satisfy all the hypotheses of Theorem \ref{representation thm}. To this end,
we also introduce  the special class of functionals $G_\e(\cdot, A): L^1(A;Y)\to [0,+\infty]$ defined as    
\begin{equation}\label{definizione di G}
       G_{\varepsilon}(u,A):=\int_{A}\int_{A}\rho_{\varepsilon}\left(y-x\right)\frac{|u(y)-u(x)|}{|y-x|}\, dx\, dy,  
\end{equation}
whose asymptotic behaviour as $\e\to 0$ will be examined in Section \ref{explicit_example}. {When $A=\Omega$, we simply write $G_\e(u)$ in place of $G_\e(u,\Omega)$.}

\noindent Note that, by (H1) and (H2), for any $A \in \mathcal{A}(\Omega)$ and $ u \in L^1(A;Y)$ it holds
\begin{equation}\label{GFG}
C_1G_\e(u,A)\leq F_\e(u,A)\leq C_2G_\e(u,A).     
\end{equation}

\section{The \texorpdfstring{$\Gamma$-}-limit of a class of reference energies}
\label{explicit_example}

Since, by \eqref{GFG}, our family of nonlocal functionals $F_\e$ are estimated by the functionals $G_\e$ defined in \eqref{definizione di G}, we first examine the asymptotic behaviour of the latter.  
Our analysis strongly relies on the $\Gamma$-convergence analysis developed in \cite{P2} for functionals of the form  \eqref{definizione di G}, that in our case are defined in $L^1(\Omega;Y)$.

\noindent We associate with the kernel
$\rho_\varepsilon:\rr^d\to[0,+\infty)$ the measure $\mu_\varepsilon$ on $S^{d-1}$
defined by
\begin{equation*}
\mu_\varepsilon(E):=\int_{R_+E}\rho_\varepsilon(z)\, dz \ \quad \ \text{for each}\ \  E\subset S^{d-1},
\end{equation*}
where $R_+E:=\{qx: q\geq 0\ \text{and}\ x \in E\}$ is the cone generated by $E$ with respect to the origin.  
Note that 
\begin{equation*}
\mu_\varepsilon(E)=
\int_E
\int_0^{+\infty}
\rho_\varepsilon(t\sigma)t^{d-1}
\,dt\,d\mathcal H^{d-1}(\sigma).
\end{equation*}
Therefore, 
\begin{equation}
\label{mu_hatK_representation}
\mu_\varepsilon
=
a_\varepsilon(\cdot)\,\mathcal H^{d-1}\lfloor_{S^{d-1}},
\qquad a_\varepsilon(\sigma):=
\int_0^{+\infty}
\rho_\varepsilon(t\sigma)t^{d-1}
\,dt.
\end{equation}
Moreover, if $\rho_\e$ is radial, then $\mu_\e=\displaystyle\frac{\mu_\e(S^{d-1})}{\sigma_{d-1}} {\cal H}^{d-1}\lfloor_{S^{d-1}}$.

If $\rho_\e$ satisfies \eqref{eq:nuclei2}, the family $(\mu_\e)$ is bounded in $\mathcal{M}(S^{d-1})$ (the space of Radon measures on $S^{d-1}$), so there exists a sequence $\e_j\to 0$ and $\mu\in \mathcal{M}({S}^{d-1})$  such that
$\mu_{\e_j}\stackrel{*}\rightharpoonup \mu$ in $\mathcal{M}(S^{d-1})$. In particular, $\mu \geq 0$ on $S^{d-1}$ and $\mu(S^{d-1})=1$.
\begin{theorem}\label{Thm_gamma_limite_esplicito}
Let $A \in \mathcal{A}^{\reg}(\Omega)$ and $G_\varepsilon(u,A)$ be defined by \eqref{definizione di G} with $\rho_\varepsilon:\rr^d\to[0,+\infty)$ satisfying \eqref{eq:nuclei3}.
Suppose that the measures $\mu_\varepsilon$ weakly$^*$ converge to a measure $\mu \in \mathcal{M}(S^{d-1})$ such that  $\spann(\supp(\mu))=\rr^d.$
Then
\begin{equation}\label{limitG}
    \Gamma(L^1)\,\text{-}\lim_{\varepsilon\rightarrow 0}G_\varepsilon(u,A)=
    \begin{cases}
    \displaystyle 
        \int_{S_u\cap A} |u^{+}-u^{-}|\varphi(\nu_u)\, d\mathcal{H}^{d-1}&\text{if}\ u\in BV(A;Y)\\
        +\infty &\text{otherwise,}
    \end{cases}
\end{equation}
  where 
\begin{equation}\label{density}
\varphi(\nu)=\displaystyle\int_{S^{d-1}}|\langle \nu, \sigma \rangle|\, d\mu(\sigma).
\end{equation}
Moreover, if $u\in BV(A;Y)$, 
\begin{equation}\label{puntuale}
\Gamma(L^1)\,\text{-}\lim_{\varepsilon\rightarrow 0}G_\varepsilon(u,A)=
\lim_{\varepsilon\rightarrow 0}G_\varepsilon(u,A).
\end{equation} 

\end{theorem}
\begin{remark}\label{misurafinita}
Under the assumptions of Theorem \ref{Thm_gamma_limite_esplicito},  we see that, for $\e$ small enough,
$$
0< c_1\le\int_{\rr^d}\rho_\e(z)\, dz\le c_2,
$$
for some positive constants $c_1,c_2$ independent on $\e$.
 Indeed,
$$
\int_{\rr^d}\rho_\e(z)\, dz=\mu_\e(S^{d-1})\to \mu(S^{d-1})\in (0,+\infty).
$$
It would not be  restrictive to assume that $\rho_\e$ satisfies \eqref{eq:nuclei2}, up to replacing $\rho_\e$ with $\tilde\rho_\e:=\displaystyle\frac{1}{\mu_\e(S^{d-1})}\rho_\e$.
\end{remark}

\begin{remark}\label{rmk:coercive}
    Given $\mu$ as in Theorem \ref{Thm_gamma_limite_esplicito}, consider the function $f_\mu: \mathbb R^{m\times d}\to [0,+\infty)$ defined by
\begin{equation}\label{fmu}
f_\mu(M):=
\int_{S^{d-1}} |M\sigma|\,d\mu(\sigma).
\end{equation}
We note that $f_\mu$ is convex, continuous, and positively one-homogeneous. Moreover, if $f_\mu(M)=0$ then $M\sigma=0$ for $\mu$-a.e. $\sigma.$ By continuity, this holds for every $\sigma\in$ supp$\,\mu$.  
Since $
\operatorname{span}(\operatorname{supp}\mu)=\mathbb R^d$, it follows that $M=0$.
Therefore, by the compactness of the unit sphere in $\mathbb R^{m\times d}$,
there exists $\alpha>0$ such that
\begin{equation}
\label{matrix_coercivity}
    \alpha |M|
    \leq
    \int_{S^{d-1}} |M\sigma|\,d\mu(\sigma)
    \qquad \forall M\in\mathbb R^{m\times d}.
\end{equation}
In particular
\begin{equation*}
0<\alpha\leq \varphi(\nu)\qquad\forall\nu\in S^{d-1}.
\end{equation*}
Note, moreover, that
$$
 \varphi(\nu)\leq \mu(S^{d-1})\qquad\forall\nu\in S^{d-1}.
$$
\end{remark}
\begin{remark}\label{rmk: rhocoercive}
If $\rho_\e$ is radial and \eqref{eq:nuclei2} is satisfied, then
$\operatorname{supp}\mu=S^{d-1}$. Hence, the spanning condition in
Theorem~\ref{Thm_gamma_limite_esplicito} is automatically satisfied,
and \eqref{matrix_coercivity} follows.

Motivated by the cone condition in \cite[Section~1.3]{P2},
we give a more general sufficient condition for
$\operatorname{span}(\operatorname{supp}\mu)=\rr^d$.
Suppose that $\mu_\e\stackrel{*}{\rightharpoonup}\mu$ as $\e\to0$.
Assume that there exist linearly independent vectors
$v_1,\dots,v_d\in\rr^d$ and $\delta>0$ such that
$$
\limsup_{\e\to0}
\int_{C_\delta(v_i)}\rho_\e(z)\,dz>0
\qquad\text{for every }i=1,\dots,d,
$$
where, for $v\in\rr^d\setminus\{0\}$,
$$
C_\delta(v):=
\left\{
w\in\rr^d\setminus\{0\}:
\left\langle\frac{v}{|v|},\frac{w}{|w|}\right\rangle>1-\delta
\right\}.
$$
Set
$$
E_\delta^i:=C_\delta(v_i)\cap S^{d-1},
\qquad i=1,\dots,d.
$$
We require $\delta$ to be sufficiently small that every choice
$\sigma_i\in\overline{E_\delta^i}$, $i=1,\dots,d$,
yields linearly independent vectors. By weak$^*$ convergence,
$$
\begin{aligned}
\mu(\overline{E_\delta^i})
&\geq \limsup_{\e\to0}\mu_\e(\overline{E_\delta^i})\\
&\geq \limsup_{\e\to0}\mu_\e(E_\delta^i)
=\limsup_{\e\to0}
\int_{C_\delta(v_i)}\rho_\e(z)\,dz>0.
\end{aligned}
$$
Thus, $\operatorname{supp}\mu$ intersects each
$\overline{E_\delta^i}$ and therefore contains a basis of $\rr^d$.
Consequently, the spanning condition holds,
and \eqref{matrix_coercivity} follows.
\end{remark}
\begin{remark}
    The previous theorem also includes the case ${\rho}_\e(z)=\displaystyle \frac{1}{\e^d}{\rho}\left(\frac{z}{\e}\right)$, 
    where $\rho\in L^1(\rr^d;[0,+\infty))$ and
$$
\int_{\rr^d}\rho(z)\,dz=1.
$$
Then
    \begin{equation*}
        \varphi(\nu)=\int_{\rr^d}\frac{\rho(z)}{|z|}|\langle \nu,z\rangle|\, dz,
    \end{equation*}
    
\end{remark}

\begin{proof}[Proof of Theorem \ref{Thm_gamma_limite_esplicito}]
\leavevmode\\[1mm]
\textbf{$\Gamma$-liminf inequality.}
Let $u_\varepsilon\in L^1(A;Y)$ be such that
$u_\varepsilon\to u
\ \text{in }L^1(A;\mathbb R^m)$ and
\begin{equation*}
\liminf_{\e \to 0} G_\varepsilon(u_\varepsilon,A)<+\infty.
\end{equation*}
We have to prove that $u\in BV(A;Y)$ and
\begin{equation}\label{liminfineq}
\liminf_{\varepsilon\to0} G_\varepsilon(u_\varepsilon,A)
\geq
\int_{S_u\cap A}
|u^+-u^-|
\int_{S^{d-1}}
|\langle \nu_u,\sigma\rangle|
\,d\mu(\sigma)
\,d\mathcal H^{d-1}.
\end{equation}
Let $\vartheta\in C_c^\infty(\mathbb R^d)$ be a nonnegative radial function 
such that
\begin{equation*}
\int_{\mathbb R^d}\vartheta(x)\,dx=1,
\qquad
\operatorname{supp}\vartheta\subset B_1,
\end{equation*}
and set
\begin{equation}
\label{def_mollifiers}
\vartheta_\delta(x):=\frac{1}{\delta^d}\vartheta\left(\frac{x}{\delta}\right),
\end{equation}
for $\delta>0$. We define
\begin{equation*}
A_\delta:=
\{x\in A:\operatorname{dist}(x,\partial A)>\delta\}
\end{equation*}
and
\begin{equation*}
    u_{\varepsilon,\delta}(x):=\vartheta_\delta\ast u_\e (x)=
\int_A \vartheta_\delta(x-y)u_\varepsilon(y)\,dy,
\qquad x\in A_\delta.
\end{equation*}
For simplicity, let us identify $G_\e(\cdot,A)$ with its extension to the whole $L^1(A;\rr^m)$. By applying Jensen's inequality and a change of variable, the following inequality holds
\begin{equation*}
    G_\varepsilon(u_\varepsilon,A)
    \geq
 G_\varepsilon(u_{\varepsilon,\delta},A_\delta),
\end{equation*}
(see \cite[Lemma 4]{P2}). Let $A'\Subset A$. For $\delta>0$ sufficiently small, we have
$A'\Subset   A_\delta$. Then
\begin{equation*}
G_\varepsilon(u_\varepsilon,A)
\geq G_\varepsilon(u_{\varepsilon,\delta},A').
\end{equation*}
Since $u_\varepsilon\to u$ in $L^1(A;\rr^m)$, we have
\begin{equation*}
u_{\varepsilon,\delta}\to u_\delta:=\vartheta_\delta\ast u
\quad\text{as }\e\to 0\quad\text{in }C^2(\overline {A'};\mathbb R^m).
\end{equation*}
Therefore, applying
\cite[Proposition 1 and Remark 7]{P2}, we get
\begin{equation}\label{limudelta}
\lim_{\varepsilon\to0}
 G_\varepsilon(u_{\varepsilon,\delta},A')
=
\int_{A'}
\int_{S^{d-1}}
|\nabla u_\delta(x)\sigma|
\,d\mu(\sigma)\,dx=\int_{A'} f_\mu(\nabla u_\delta)\, dx
\end{equation}
where $f_{\mu}$ is defined as in \eqref{fmu}.
Consequently, by \eqref{matrix_coercivity},
\begin{equation*}
    \begin{split}
\liminf_{\varepsilon\to0} G_\varepsilon(u_\varepsilon,A)
&\geq
\liminf_{\varepsilon\to0}
 G_\varepsilon(u_{\varepsilon,\delta},A')\\
&=
\int_{A'}
\int_{S^{d-1}}
|\nabla u_\delta(x)\sigma|
\,d\mu(\sigma)\,dx\\
&\geq \alpha
\int_{A'}|\nabla u_\delta(x)|\,dx.
\end{split}
\end{equation*}
Since $u_\delta\to u$ in $L^1(A';\rr^m)$, 
the lower
semicontinuity of the total variation yields 
\begin{equation*}
|Du|(A')
\leq
\liminf_{\delta\to0}\int_{A'}|\nabla u_\delta(x)|\, dx
\leq
\frac1\alpha
\liminf_{\varepsilon\to0}G_\varepsilon(u_\varepsilon,A)<+\infty,
\end{equation*}
from which we infer that $u\in BV(A; Y)$ by the arbitrariness of
$A'\Subset A$. Since $f_\mu$ is convex and positively one-homogeneous, the $L^1$-lower-semicontinuous envelope of the functional
\begin{equation*}
F_\mu(u)=\begin{cases} 
\displaystyle \int_{A'} f_\mu (\nabla u)\, dx,& u\in C^2(\overline {A'};\rr^m)\\\\
+\infty,& \text{otherwise},
    \end{cases}
\end{equation*}
is given by 
\begin{equation*}
\overline F_\mu(u)=\begin{cases}
\displaystyle \int_{A'} f_\mu\left(\frac{d Du}{d |D u|}\right)\,d |Du|=\int_{S^{d-1}}
|Du\,\sigma|(A')\,d\mu(\sigma)& u\in BV(A';\rr^m)\\\\
+\infty,& \text{otherwise},
    \end{cases}
\end{equation*}
where $D u$ denotes the distributional derivative of $u$.
Hence, by \eqref{limudelta}, passing to the limit as
$\delta\to0$, we obtain
\begin{equation}
\label{lwr_bd}
\liminf_{\varepsilon\to0} 
G_\varepsilon(u_\varepsilon,A)
\geq
\int_{S^{d-1}}
|Du\,\sigma|(A')\,d\mu(\sigma).
\end{equation}
Since $u\in BV(A;Y)$, we have
$Du=(u^+-u^-)\otimes\nu_u\,
\mathcal H^{d-1}\llcorner S_u$.
Therefore, for every $\sigma\in S^{d-1}$, {}
\begin{equation*}
|Du\,\sigma|(A')
=
\int_{S_u\cap A'}
|u^+-u^-|
|\langle\nu_u,\sigma\rangle|
\,d\mathcal H^{d-1}.
\end{equation*}
Hence, by Fubini's theorem, 
\begin{equation*}
    \int_{S^{d-1}}
|Du\,\sigma|(A')\,d\mu(\sigma)
=
\int_{S_u\cap A'}
|u^+-u^-|
\int_{S^{d-1}}
|\langle\nu_u,\sigma\rangle|
\,d\mu(\sigma)
\,d\mathcal H^{d-1}.
\end{equation*}
Combining this with \eqref{lwr_bd}, we obtain 
\begin{equation*}
    \liminf_{\e \to 0}G_\e(u_\e,A)\geq \int_{S_u\cap A'}
|u^+-u^-|
\int_{S^{d-1}}
|\langle\nu_u,\sigma\rangle|
\,d\mu(\sigma)
\,d\mathcal H^{d-1},
\end{equation*}
from which we get \eqref{liminfineq}, by the arbitrariness of $A'\Subset A$. 

\medskip
\noindent
\textbf{$\Gamma$-limsup inequality and pointwise limit.}
It suffices to prove that for every $u\in BV(A;Y)$,
\begin{equation}\label{limsupineq}
    \limsup_{\varepsilon\to0} G_\varepsilon(u,A)
\leq
\int_{S_u\cap A}
|u^+-u^-|
\int_{S^{d-1}}
|\langle \nu_u,\sigma\rangle|
\,d\mu(\sigma)
\,d\mathcal H^{d-1}.
\end{equation}
Let then $u\in BV(A;Y)$ and consider an extension $\bar u \in BV(\mathbb{R}^d;\mathbb{R}^m)$ such that  
$|D\bar u|(\partial A)=0$ (see \cite[Proposition 3.21]{AFP}). Consider the function $\Phi:\sigma \in S^{d-1}\mapsto |Du\,\sigma|(A) $. Since $u \in BV(A;Y)$,
\begin{equation}
\label{eq:Phi}
    \Phi(\sigma)=\int_{S_u\cap A}|u^+-u^-|\,|\langle\nu_u,\sigma\rangle|\,d\mathcal H^{d-1}.
\end{equation}
Note that 
$\Phi$ is continuous on $S^{d-1}$. For $q>0$ set
$$
W_q:=\{y\in\rr^d:\ \operatorname{dist}(y,A)<q\},\qquad
\omega(q):=|D\bar u|\big(\overline{W}_{2q}\setminus A\big).
$$
The function $\omega$ is nondecreasing and, since $\bigcap_{q>0}\big(\overline{W}_{2q}\setminus A\big)=\overline A\setminus A=\partial A$
and $|D\bar u|(\mathbb R^d)<+\infty$,
\begin{equation}
\label{eq:omega}
\lim_{q\to0^+}\omega(q)=|D\bar u|(\partial A)=0.
\end{equation}
By the change of variables $z=y-x$,
we have
\begin{equation*}
\begin{split}
G_\varepsilon(u,A)
&=
\int_A\int_A
\rho_\varepsilon(y-x)
\frac{|u(y)-u(x)|}{|y-x|}
\,dx\,dy \\
&=
\int_{\mathbb R^d}
\int_{A\cap (A-z)}\rho_\varepsilon(z)
\frac{|u(x+z)-u(x)|}{|z|}
\,dx\,dz.
\end{split}
\end{equation*}
For $z\neq0$ we define
\begin{equation*}
    g(z):=
\int_{A\cap (A-z)}
\frac{|u(x+z)-u(x)|}{|z|}
\,dx,
\end{equation*}
and we set $g(0)=0.$ Then, we may rewrite the energy as 
\begin{equation}
\label{energy_g_epsilon}
G_\varepsilon(u,A)
=
\int_{\mathbb R^d} \rho_\varepsilon(z)g(z)
\,dz,
\end{equation}
and we claim that 
\begin{equation}
\label{claim1_limsup}
g(z)\leq \Phi\left(\frac{z}{|z|}\right)+\omega(|z|),\ \quad \ \forall z\neq 0.
\end{equation}
Indeed, fix $z \neq 0$ and set $A_z:=A\cap (A-z)$. For every $x \in A_z$ and $t \in [0,1]$ it holds 
$$
\operatorname{dist}(x+tz,A)\leq t|z|<2|z|,
$$
that is 
\begin{equation}
\label{inclusion_sets_limsup}
    A_z+tz\subset W_{2|z|}.
\end{equation} Since $\bar u=u$ in $A$ and
$x,x+z\in A$ for every $x\in A_z$, it holds
\begin{equation}
\label{controllo_traslazioni}
    |z|\,g(z)=\int_{A_z}|\bar u(x+z)-\bar u(x)|\,dx\le|z|\, \left |D\bar u\,\frac{z}{|z|}\right |\big(\overline{W}_{2|z|}\big).
\end{equation}
Although this estimate is standard, we include a proof for completeness.
Set $\bar u_\delta:= \vartheta_\delta \ast\bar u$, where $\vartheta_\delta$ are the mollifiers introduced in \eqref{def_mollifiers}. By \eqref{inclusion_sets_limsup}, the change of variable $y=x+tz$ and Fubini's theorem it holds
\begin{equation}
\label{stima1_traslazione_mollificata}
\int_{A_z}|\bar u_\delta(x+z)-\bar u_\delta(x)|\,dx
\le|z|\int_0^1\int_{A_z}\left|\nabla\bar u_\delta(x+tz)\,\frac{z}{|z|}\right|\,dx\,dt
\le|z|\int_{W_{2|z|}}\left|\nabla\bar u_\delta(y)\,\frac{z}{|z|}\right|\,dy .
\end{equation}
Since $\nabla\bar u_\delta=\vartheta_\delta \ast D\bar u$ and $\vartheta_\delta\ge0$, we have
$$\left|\nabla\bar u_\delta(y)\,\frac{z}{|z|}\right|\le\int_{\mathbb R^d}\vartheta_\delta(y-w)\,d\left|D\bar u\,\frac{z}{|z|}\right|(w)$$ for every $y$; integrating
over $W_{2|z|}$ and using Fubini's theorem,
\begin{equation}
\label{stima2_traslazione_mollificata}
\int_{W_{2|z|}}\left|\nabla\bar u_\delta\,\frac{z}{|z|}\right|\,dy
\le\int_{\mathbb R^d}\Big(\int_{W_{2|z|}}\vartheta_\delta(y-w)\,dy\Big)\,d\left |D\bar u\,\frac{z}{|z|}\right|(w)
\le\left|D\bar u\,\frac{z}{|z|}\right|\big(W_{2|z|}+B_\delta\big).
\end{equation}
 Hence, combining \eqref{stima1_traslazione_mollificata} and \eqref{stima2_traslazione_mollificata}
$$
\int_{A_z}|\bar u_\delta(x+z)-\bar u_\delta(x)|\,dx\le|z|\,\left|D\bar u\,\frac{z}{|z|}\right|\big(W_{2|z|}+B_\delta\big).
$$
We now let $\delta\to0$. Since $\bar u\in L^1(\mathbb R^d;\mathbb R^m)$, we have $\bar u_\delta\to\bar u$ in
$L^1(\mathbb R^d;\mathbb R^m)$, and therefore the left-hand side converges to
$$\int_{A_z}|\bar u(x+z)-\bar u(x)|\,dx.$$ On the other hand, the sets $W_{2|z|}+B_\delta$ decrease to
$$\bigcap_{\delta>0}(W_{2|z|}+B_\delta)=\overline W_{2|z|}$$ and $|D\bar u|(\mathbb R^d)<+\infty$, so that
$$\left|D\bar u\,\frac{z}{|z|}\right|(W_{2|z|}+B_\delta)\to\left|D\bar u\,\frac{z}{|z|}\right|(\overline W_{2|z|}).$$
This proves 
\eqref{controllo_traslazioni}. Moreover, since $D\bar u\lfloor_A=Du$,
\begin{equation}
\label{ulteriore_controllo}
    \left |D\bar u\,\frac{z}{|z|}\right |\big(\overline{W}_{2|z|}\big)=\left |D u\,\frac{z}{|z|}\right |(A)+\left |D\bar u\,\frac{z}{|z|}\right |\big(\overline{W}_{2|z|}\setminus A\big)
\le\Phi\left(\frac{z}{|z|}\right)+\omega(|z|).
\end{equation}
Combining \eqref{controllo_traslazioni} and \eqref{ulteriore_controllo} and dividing by $|z|$, we get \eqref{claim1_limsup}. We also have 
\begin{equation}
\label{eq:g_trivial}
g(z)\leq\frac{C\,|A|}{|z|}\qquad\text{for every }z\ne0.
\end{equation}
By Remark \ref{misurafinita} there exist $C>0$ and $\varepsilon_0>0$ such that
$$\int_{\mathbb R^d}\rho_\varepsilon(z)\,dz\le C,$$ 
for every $\varepsilon<\varepsilon_0$. Fix $\eta>0$.
By \eqref{energy_g_epsilon}, \eqref{claim1_limsup}, the monotonicity of $\omega$ and \eqref{eq:g_trivial},
for $\varepsilon<\varepsilon_0$ we get
\begin{equation}
\label{penultimo_controllo}
    G_\varepsilon(u,A)\le\int_{\mathbb R^d}\rho_\varepsilon(z)\,\Phi\Big(\frac{z}{|z|}\Big)dz+C\,\omega(\eta)
+\frac{C\,|A|}{\eta}\int_{|z|>\eta}\rho_\varepsilon(z)\,dz .
\end{equation}
Passing to polar coordinates and using \eqref{mu_hatK_representation}, 
\begin{equation}
\label{stima_primo_addendo}
\int_{\mathbb R^d}\rho_\varepsilon(z)\,\Phi\Big(\frac{z}{|z|}\Big)dz
=\int_{S^{d-1}}\Phi(\sigma)\int_0^{+\infty}\rho_\varepsilon(t\sigma)t^{d-1}\,dt\,d\mathcal H^{d-1}(\sigma)
=\int_{S^{d-1}}\Phi(\sigma)\,d\mu_\varepsilon .
\end{equation}
Since $\Phi$ is continuous and $\mu_\varepsilon\overset{*}{\rightharpoonup}\mu$, combining  \eqref{eq:nuclei3}, \eqref{penultimo_controllo}, \eqref{stima_primo_addendo} we get
$$
\limsup_{\varepsilon\to0}G_\varepsilon(u,A)\le\int_{S^{d-1}}\Phi(\sigma)\,d\mu+C\,\omega(\eta).
$$
Letting $\eta\to0^+$, by  \eqref{eq:Phi}, \eqref{eq:omega}, together with Fubini's theorem,
$$
\limsup_{\varepsilon\to0}G_\varepsilon(u,A)\le\int_{S^{d-1}}\Phi(\sigma)\,d\mu
=\int_{S_u\cap A}|u^+-u^-|\int_{S^{d-1}}|\langle\nu_u,\sigma\rangle|\,d\mu(\sigma)\,d\mathcal H^{d-1},
$$
which is \eqref{limsupineq}.
\end{proof}

\noindent By Theorem \ref{Thm_gamma_limite_esplicito} and Remarks \ref{rmk:coercive} and \ref{rmk: rhocoercive}, we immediately deduce the following corollary.
\begin{corollary}\label{liminfG}
Let $G_\varepsilon(u,A)$ be defined by \eqref{definizione di G} with $\rho_\varepsilon:\rr^d\to [0,+\infty)$  satisfying \eqref{eq:nuclei1}, \eqref{eq:nuclei2} and \eqref{eq:nuclei3}. If $\displaystyle \Gamma\text{-}\liminf_{\e\to 0} G_\e(u,A)$ is finite, then $u\in BV(A;Y)$. Moreover, there exist two positive constants $\alpha$ and $\beta$ such that for every $A\in {\cal A}^{reg}(\Omega)$ and $u\in BV(A;Y)$
$$
\alpha\,{\cal H}^{d-1}(S_u\cap A)\leq \Gamma(L^1)\,\text{-}\liminf_{\e\to 0} G_\e(u,A)\leq \Gamma(L^1)\,\text{-}\limsup_{\e\to 0} G_\e(u,A)\leq \beta\, {\cal H}^{d-1}(S_u\cap A).
$$
\end{corollary}

We conclude this section with four explicit examples that show that the limiting surface energy density in Theorem \ref{Thm_gamma_limite_esplicito}  may be either isotropic or anisotropic. The first two examples involve radial kernels, which yield isotropic
densities, while the last two show how anisotropic densities can arise.

\begin{example}[Scaled fractional perimeters]
\label{ex_frac_per}
 We first recall a notion of fractional perimeter appropriately defined in order to deal with nonlocal contributions only inside the target set $\Omega$. 
        
Let $E\subset \Omega$ be a measurable set and $0<s<1$. The fractional $s$-perimeter  of $E$ inside $\Omega$ is denoted by $P_s(E,\Omega)$ and it is defined as follows:
        $$P_s(E,\Omega):=\int_{E}\int_{E^c\cap \Omega}\frac{1}{|y-x|^{d+s}}\, dx\, dy.$$
       Note that $P_s(E,\Omega)$ is half of the fractional Sobolev space seminorm $|\chi_E|_{W^{s,1}(\Omega)},$ where $\chi_E$ denotes the characteristic function of $E$. We recall that if $E$ is a set of finite perimeter in $\Omega$, then $P_s(E,\Omega)$ is also finite, since in general $BV(\Omega)\subset W^{s,1}(\Omega)$ for any $s \in (0,1)$. 
     
     The asymptotic analysis as $s\to 1^-$ of 
    $(1-s)P_s(E,\Omega)$ has been performed  in \cite{ADM}, where the authors prove that the $\Gamma$-limit is the functional $\omega_{d-1}P(E,\Omega)$, with $P(E,\Omega)$ denoting the classical De Giorgi's perimeter. This family of scaled perimeters can be recast in our notation and the result in \cite{ADM} is covered by Theorem \ref{Thm_gamma_limite_esplicito}. Indeed,
let $s=1-\varepsilon$ and set 
   \begin{equation*}
   \begin{split}
       G_\e(u)= \e P_{1-\e}(\{u=1\},\Omega)=\frac{\e}{2}\int_{\Omega}\int_{\Omega} \frac{|u(y)-u(x)|}{|y-x|^{d+1-\varepsilon}}\, dx\, dy.
    \end{split}
   \end{equation*}
   Note that, setting $D=\textrm{diam}(\Omega)$, $G_\e(u)$ can be written as
   \begin{equation*}
   \begin{split}
       G_\e(u)=\frac\e2\int_{\Omega}\int_{\Omega} \frac{|u(y)-u(x)|}{|y-x|^{d+1-\varepsilon}}\chi_{[0,D]}(|y-x|)\, dx\, dy.
    \end{split}
   \end{equation*}
    which is of the form \eqref{definizione di G} with
   \begin{equation*}
       \rho_\e(z)=\frac\e2\frac{1}{|z|^{d-\varepsilon}}\chi_{[0,D]}(|z|).
   \end{equation*}
   Note that, for every $\eta>0$,
   \begin{equation*}
       \int_{|z|>\eta} \rho_{\varepsilon
       }(z)\, dz=\frac{\e\sigma_{d-1}}{2}\int_\eta^D t^{\e-1}\, dt=\frac{\sigma_{d-1}}{2} t^{\varepsilon}\bigg|^{D}_{\eta}=\frac{\sigma_{d-1}}{2}(D^{\varepsilon}-{\eta}^{\varepsilon})\rightarrow0,
   \end{equation*}
hence \eqref{eq:nuclei3} is satisfied. Moreover,
$$
\mu_\e(S^{d-1})=\int_{\rr^d}\rho_\e(z)\, dz=
\frac{\sigma_{d-1}}{2} t^{\varepsilon}\bigg|^{D}_{0}=\frac{\sigma_{d-1}}{2} D^\e.
$$
Hence,
$$
\mu_\e=\displaystyle\frac{D^\e}{2} {\cal H}^{d-1}\lfloor_{S^{d-1}}\stackrel{*}\rightharpoonup  \mu=\frac{1}{2} {\cal H}^{d-1}\lfloor_{S^{d-1}}
$$
and, by \eqref{density},
$$
\varphi(\nu)=\frac12\int_{S^{d-1}}|\langle\nu,\sigma\rangle|\, d{\cal H}^{d-1}(\sigma)=\omega_{d-1}.
$$
\end{example}

\begin{example}[Scaled volume of an $\e$-neighborhood of the interface]
\label{volume}

Set 
   \begin{equation*}
         G_\e(u)=\frac{1}{\e^{d+1}}\int_{\{(x,y)\in \Omega \times \Omega:\, |y-x|\leq \e\}}|u(y)-u(x)|\, dx\, dy.
   \end{equation*}
  Note that $G_\e$ is of the form \eqref{definizione di G} with
   $$
       \rho_\e(z)=\frac{|z|}{\e^{d+1}}\chi_{[0,\e]}(|z|).
   $$
In this case \eqref{eq:nuclei3} is trivially satisfied and   
$$
\mu_\e(S^{d-1})=\int_{\rr^d}\rho_\e(z)\, dz
=\frac{\sigma_{d-1}}{\e^{d+1}}\int_0^\e t^d\, dt= \frac{\sigma_{d-1}}{d+1}
$$
Hence,
$$
\mu_\e= \frac{1}{d+1}{\cal H}^{d-1}\lfloor_{S^{d-1}}= \mu
$$
and, by \eqref{density},
$$
\varphi(\nu)=\frac{1}{d+1}\int_{S^{d-1}}|\langle\nu,\sigma\rangle|\, d{\cal H}^{d-1}(\sigma)=\frac{2\omega_{d-1}}{d+1}.
$$
\end{example}

\begin{example}[Anisotropic finite-range interactions]
\label{ex_anisotropy}
   Let $d\ge 2$, $C>0$, and consider $\displaystyle \rho_\varepsilon(z)=\frac{1}{\e^d}\chi_{B_\e}(z)a\left(\frac{z}{|z|}\right)$, where $B_\e$ is the ball centered at the origin with radius $\e$ and
    $$a(\sigma):=1+C|\langle \sigma, e_1\rangle|.
    $$
    In this case
    $$
    \mu_\e=\mu=\frac{1}{d}
    a(\cdot){\cal H}^{d-1}\lfloor_{S^{d-1}},
    $$
    thus
    $$
\varphi(\nu)=\frac1d\int_{S^{d-1}}\left(1+C|\langle \sigma, e_1\rangle|\right)|\langle\sigma,\nu\rangle|\, d{\cal H}^{d-1}(\sigma).
    $$
    A straightforward computation shows that  $\varphi(e_1)\neq\varphi(e_2)$, and hence the resulting surface energy is anisotropic.
\end{example}

\begin{example}[Anisotropic fractional perimeters]\label{ex:anisotropic}
For $s\in(0,1)$, let $k_s\colon\mathbb{R}^d\setminus\{0\}\to(0,+\infty)$ be a measurable
function satisfying, for some constant $c\ge1$ independent of $s$,
\begin{align}
&k_s(-z)=k_s(z) &&\text{for every } z\neq0, \label{k_even}\\
&k_s(\lambda z)=\lambda^{-d-s}k_s(z) &&\text{for every } \lambda>0,\ z\neq0, \label{k_hom}\\
&\frac{1}{c}\,\frac{1}{|z|^{d+s}}\le k_s(z)\le \frac{c}{|z|^{d+s}} &&\text{for every } z\neq0, \label{k_bounds}
\end{align}
and assume that $k_s\to k$ pointwise in $\mathbb{R}^d\setminus\{0\}$ as $s\to1^-$, for some measurable
function $k:\mathbb R^d\setminus\{0\}\to(0,+\infty)$.
For $u=\chi_E\in L^1(\Omega;\{0,1\})$, the anisotropic $s$-perimeter of $E$ associated
with $k_s$, restricted to the interactions inside $\Omega$, is
$$
P_{k_s}(u):=\int_{E\cap\Omega}\int_{E^c\cap\Omega}k_s(y-x)\,dx\,dy
=\frac12\int_\Omega\int_\Omega k_s(y-x)\,|u(y)-u(x)|\,dx\,dy ,
$$
where the equality follows from \eqref{k_even}. These energies have been studied in
\cite{L14,F25}. In this example we show that they
fall within our framework.

Set $\varepsilon:=1-s$ and $D:=\operatorname{diam}(\Omega)$. Since $|y-x|\le D$ for every
$x,y\in\Omega$, the rescaled energies $\varepsilon P_{k_{1-\varepsilon}}$ are of the
form \eqref{rescale_energy}, i.e.\ $\varepsilon P_{k_{1-\varepsilon}}=G_\varepsilon$, with
$$
\rho_\varepsilon(z)=\frac{\varepsilon}{2}\,k_{1-\varepsilon}(z)\,|z|\,\chi_{[0,D]}(|z|),
$$
which is not radial in general. By \eqref{k_hom}, for every $\sigma\in S^{d-1}$ we have
$$
a_\varepsilon(\sigma)=\int_0^{+\infty}\rho_\varepsilon(t\sigma)\,t^{d-1}\,dt
=\frac{\varepsilon}{2}\,k_{1-\varepsilon}(\sigma)\int_0^D t^{\varepsilon-1}\,dt
=\frac{D^\varepsilon}{2}\,k_{1-\varepsilon}(\sigma),
$$
so that, by \eqref{mu_hatK_representation},
$$
\mu_\varepsilon=\frac{D^\varepsilon}{2}\,k_{1-\varepsilon}(\cdot)\,\mathcal H^{d-1}\lfloor_{S^{d-1}}.
$$
In the same way, for every $\eta\in(0,D)$, by \eqref{k_bounds} we get
$$
\int_{\{|z|>\eta\}}\rho_\varepsilon(z)\,dz
=\frac{D^\varepsilon-\eta^\varepsilon}{2}\int_{S^{d-1}}k_{1-\varepsilon}(\sigma)\,d\mathcal H^{d-1}(\sigma)
\le\frac{c\,\sigma_{d-1}}{2}\,(D^\varepsilon-\eta^\varepsilon)\to0
\quad\text{as }\varepsilon\to0,
$$
hence \eqref{eq:nuclei3} is satisfied. Moreover, by \eqref{k_bounds} we have
$1/c\le k_{1-\varepsilon}\le c$ on $S^{d-1}$ for every $\varepsilon$; since
$k_{1-\varepsilon}\to k$ pointwise and $D^\varepsilon\to1$, the Dominated Convergence
Theorem yields
$$
\mu_\varepsilon\overset{*}{\rightharpoonup}\mu=\frac12\,k\,\mathcal H^{d-1}\lfloor_{S^{d-1}}
\qquad\text{as }\varepsilon\to0 .
$$
Since $k\ge1/c$ on $S^{d-1}$,
$\operatorname{supp}\mu=S^{d-1}$ and the spanning condition of
Theorem \ref{Thm_gamma_limite_esplicito} is satisfied. 

Applying
Theorem \ref{Thm_gamma_limite_esplicito} with $A=\Omega$, and computing \eqref{density}
with the measure $\mu$ above, we conclude that
the $\Gamma$-limit of $\e P_{k_{1-\e}}$ is of the form \eqref{limitG} with
\begin{equation}
\label{anisotropic_fractional_kernels}
\varphi(\nu)=\frac12\int_{S^{d-1}}|\langle\nu,\sigma\rangle|\,k(\sigma)\,d\mathcal H^{d-1}(\sigma).
\end{equation}
Note that we recover the $\Gamma$-convergence result stated in  \cite[Theorem 1.3]{F25}. Indeed, in that theorem  the energy density of the limit energy is
given by
$$
\frac{d+1}{2}\int_{B_k}|\langle x,\nu\rangle|\,dx,$$
with $B_k:=\{x\in\mathbb R^d\setminus\{0\}:\ k(x)>1\}\cup\{0\}$. By \eqref{k_hom} the limit kernel satisfies $k(\lambda z)=\lambda^{-d-1}k(z)$, so that 
$$B_k=\{t\sigma: \sigma \in S^{d-1}, 0\leq t<k(\sigma)^{\frac{1}{d+1}}\};$$
hence, passing to polar coordinates in the previous formula, we get \eqref{anisotropic_fractional_kernels}.

Finally, for $k_s(z)=|z|^{-d-s}$ we have $k\equiv1$ on $S^{d-1}$, and we recover
Example~\ref{ex_frac_per}. 
\end{example}

\section{Proof of Theorem \ref{main thm}}
\label{Proof main thm}

The next two propositions establish that, under hypotheses (H1) and (H2), for every $A\in \mathcal{A}^{\reg}(\Omega)$, $F'(u,A)$ and $F''(u,A)$ are finite only on $BV(A;Y)$ and satisfy hypothesis $(iv)$ of Theorem \ref{representation thm}. We omit the proofs, since they are a straightforward  consequence of  \eqref{GFG}  and Corollary \ref{liminfG}.

\begin{proposition}
\label{prop controllo dal basso}
    Let $K_{\varepsilon}(x,y,u,v)$ satisfy {\rm (H1)} and let $A \in \mathcal{A}(\Omega)$. If $u\in L^{1}(\Omega;Y)$ is such that $F'(u,A)$ is finite, then $u \in BV(A;Y)$ and 
    \begin{equation*}
        F'(u,A)\geq  c_1\,\mathcal{H}^{d-1}(S_u\cap A),
    \end{equation*}
    for some positive constant $c_1$ independent of $u$ and $A$.
    
\end{proposition}

\begin{proposition}
\label{prop controllo dall'alto}
Let $K_{\varepsilon}(x,y,u,v)$ satisfy {\rm (H2)} and let $A\in\mathcal A^{\reg}(\Omega)$. Then, for every
$u\in BV(A;Y)$, it holds
\begin{equation*}
F''(u,A)
\leq c_2\, \mathcal H^{d-1}(S_u\cap A),
\end{equation*}
for some positive constant $c_2$ independent of $u$ and $A$.
\end{proposition}

\noindent In the following propositions, we show that $F''(\cdot,\cdot)$ satisfies
hypotheses $(i)$ and $(ii)$ of Theorem \ref{representation thm}.

\begin{proposition}
\label{Prop subadditivity}
    Let $K_{\varepsilon}(x,y,u,v)$ satisfy {\rm (H1)}-{\rm (H2)} and let $A,B,A' \in \mathcal{A}^{\reg}(\Omega)$ be given with $A'\Subset A$. Then for any $u \in BV(\Omega;Y)$
    \begin{equation*}
        F''(u,A'\cup B)\leq F''(u,A)+F''(u,B).
    \end{equation*}
\end{proposition}
    \begin{proof}
         Without loss of generality we may assume $F''(u,A)$ and $F''(u,B)$ are finite. Let $u_{\varepsilon},v_{\varepsilon}$ be such that $u_{\varepsilon}\rightarrow u$, $v_{\varepsilon}\rightarrow u$ in $L^{1}(\Omega;\rr^m)$ and
        \begin{equation*}
            \limsup_{\varepsilon\rightarrow 0}F_{\varepsilon}(u_{\varepsilon},A)=F''(u,A),\quad  \limsup_{\varepsilon\rightarrow 0}F_{\varepsilon}(v_{\varepsilon},B)=F''(u,B).
        \end{equation*}
         Recall that the distance function $\operatorname{dist}(x,A')$ is 1-Lipschitz. For $t \in [0,\bar{d}]$, we define 
        \begin{equation*}
            A^t:=\{x \in A: \dist(x,A')\leq t\},
        \end{equation*}
      and the family of functions 
        \begin{equation*}
            \hat w_{\e,t}=\chi_{
            {A}^{t}}u_\e+(1-\chi_{
            {A}^{t}})v_\e.
            \end{equation*}
            Note that
            \begin{equation}
        \label{controllo sui moduli dei gradienti}
        \begin{split}
            |D^1_z \hat w_{\e,t}(x)|\leq|D^1_{z}u_\e(x)|+|D^1_{z}v_\e(x)|+|D^1_{z}\chi_{A^t}(x)||u_\e(x)-v_\e(x)|,
        \end{split}
        \end{equation}
        where for a generic function $w:\Omega\to\rr^m$ and $z\in\rr^d\setminus\{0\}$, we adopt the notation
        $$
        D^1_z  w (x):=\frac{w(x+z)-w(x)}{|z|}.
        $$
        Let
         $\bar{d}:=\operatorname{dist}( {A'},A^{c})$,
         $\eta\in (0, \bar{d})$  and set 
        $   \displaystyle N_{\eta}:=\left\lfloor\frac{\bar{d}}{\eta}\right\rfloor.$
     Given $\varphi:\rr\rightarrow [0,1]$ defined by
        \begin{equation*}
            \varphi(t):=
            \begin{cases}
                0 &\text{if}\ t\leq 0\\
                t &\text{if}\ 0<t<1\\
                1 &\text{if}\  t\geq 1
            \end{cases}
        \end{equation*} 
        and $k \in \{1,\dots,N_\eta-1\}$, let $\varphi_\eta^k(x)$ be the cut-off function between $A^{k\eta}$ and $A^{(k+1) \eta}$ defined as 
        \begin{equation*}
            \varphi_\eta^k(x):=\varphi\left(\frac{\operatorname{dist}(x,A')}{\eta}-k\right).
        \end{equation*} Note that 
         $\varphi_\eta^k\equiv 0$ on $A^{k\eta}$,
         $\varphi_\eta^k\equiv 1$ on $A\setminus A^{(k+1) \eta}$ and 
            \begin{equation*}
                \|\nabla \varphi_\eta^k\|_{\infty}\leq \frac{\|\nabla \operatorname{dist}(x,A')\|_{\infty}}{\eta}=\frac{1}{\eta},
            \end{equation*}
        which yields that for every $z \in \rr^d\setminus\{0\}$
        \begin{equation}
            \label{eq:diff finite}
            |D^1_z \varphi^k_\eta|\leq \|\nabla \varphi_\eta^k\|_\infty\leq \frac{1}{\eta}.
        \end{equation}
    Moreover, by the definition of $\varphi^k_\eta(x)$, for every  $ t\in (0,1)$ we have 
        \begin{equation*}
        \begin{split}
            A^{\eta(t+k)}&=\{x \in A: \dist(x,A')\leq\eta(t+k)\}
            =\{x \in A:\varphi_\eta^k(x)\leq t\},\\
            \partial A^{\eta(t+k)}&=\{x \in A: \dist(x,A')=\eta(t+k)\}
            =\{x \in A:\varphi_\eta^k(x)=t\}.
        \end{split}
        \end{equation*}
        By the Coarea formula, the fact that $|\nabla \operatorname{dist}(x, A')|=1$ and the change of variable $t=\frac{s-k\eta}{\eta}$, we get 
        \begin{equation}
        \label{eq:cf distanza}
            \begin{split}
                |A^{(k+1)\eta}\setminus A^{k\eta}|&=\int_{\{k\eta \leq \operatorname{dist}(x, A')\leq (k+1)\eta\}}|\nabla \operatorname{dist}(x,A')|\, dx\\
                &=\int_{k\eta}^{(k+1)\eta}\mathcal{H}^{d-1}(\partial A^s)\, ds\\
                &=\eta \int_0^1\mathcal{H}^{d-1}(\partial A^{\eta(t+k)})\, dt.
            \end{split}
        \end{equation}
        We set
    $g_\eta^k(t):=\mathcal{H}^{d-1}(\partial A^{\eta (t+k)})$ and 
        \begin{equation*}
            V^{k}_\eta:=\left\{t \in (0,1):g^k_\eta(t) >2\int_{0}^1g^k_\eta(s)\, ds\right\}, \quad
            E^k_\eta:=(0,1)\setminus V^k_\eta.
        \end{equation*}
        We note that 
        \begin{equation*}
            \int_{0}^1
g^k_\eta(s)\, ds\geq\int_{V^k_\eta}g^k_\eta(s)\, ds\geq |V^k_\eta|\,2\int_{0}^1g^k_\eta(s)\, ds.        \end{equation*}
        Hence, we have that $|V^k_\eta|\leq 1/2$ and $|E_\eta^k|\geq 1/2$. By \eqref{eq:cf distanza},  $\forall t\in E^k_\eta$ we get
        \begin{equation}
            \label{conseguenza cf distanza}
            \mathcal{H}^{d-1}(\partial A^{\eta(t+k)})\leq \frac{2}{\eta}|A^{(k+1)\eta}\setminus A^{k\eta}|.
        \end{equation}
        We define
$$
C^{k}_\eta:=\left(A^{(k+1)\eta }\setminus A^{k\eta }\right) \cap B
\qquad\text{and}\qquad
S_\eta^k:=(C_\eta^k+B_{\eta }) \cap B=\{x \in B: \dist(x, C^k_\eta)<\eta\},
$$
where $B_\eta$ is the open ball centered at the origin with radius $\eta>0$.  Note that 
        \begin{equation*}
\int_{0}^{1}|D_z^1\chi_{\{\varphi_\eta^k\leq t\}}(x)|\, dt=|D^1_z \varphi_\eta^k(x)|.
        \end{equation*}
Hence we have
\begin{equation}\label{defh}
\int_{ |z|\leq \eta} \rho_\e(z)\int_{S_\eta^k}
|D^1_{z}\varphi^{k}_\eta(x)||u_\e(x)-v_\e(x)|\, dx\, d{z}
=
\int_0^1 h^k_{\e,\eta}(t)\, dt,
\end{equation}
where
$$
h_{\e,\eta}^k(t):=\int_{|z|\leq \eta}{ \rho_\e(z)}\int_{S_\eta^k}
|D^1_{z}\chi_{\{\varphi^k_\eta \leq t\}}(x)||u_\e(x)-v_\e(x)|\, dx\, dz.
$$
Since $h_{\e,\eta}^k\geq 0$ and $|E_\eta^k|\geq 1/2$, there exists
$t_{\e,\eta}^k \in E_\eta^k$ such that
\begin{equation}\label{stimah}
h_{\e,\eta}^k(t_{\e,\eta}^k)
\leq
\frac{1}{|E_\eta^k|}
\int_{E_\eta^k} h_{\e,\eta}^k(t)\, dt
\leq
2\int_0^1 h_{\e,\eta}^k(t)\, dt.
\end{equation}
Setting
$$
{\overline A}^k_{\e,\eta}:=A^{\eta (t^{k}_{\e,\eta}+k)}=\{ x \in A: \varphi^k_{\eta}(x)\leq t^k_{\e,\eta}\},
$$
by \eqref{defh} and \eqref{stimah} we obtain
\begin{equation}
\label{eq:stima 2}
\begin{split}
\int_{|z|\leq \eta}&{\rho_{\e}(z)}\int_{S_\eta^k}
|D^1_{z}\chi_{\overline{A}^k_{\e,\eta}}(x)||u_\e(x)-v_\e(x)|\, dx\, d{ z}
\leq\\
&\leq 
2 \int_{|z|\leq \eta}\rho_{\e}(z)\int_{S_\eta^k}
|D^1_{z}\varphi^{k}_\eta(x)||u_\e(x)-v_\e(x)|\, dx\, dz.
\end{split}
\end{equation}
Moreover, since $t_{\e,\eta}^k \in E_\eta^k$, by \eqref{conseguenza cf distanza} we also have
\begin{equation}
\label{eq:stima 1}
\mathcal{H}^{d-1}(\partial \overline{A}^k_{\e,\eta})
\leq
\frac{2}{\eta}|A^{(k+1)\eta}\setminus A^{k\eta }|.
\end{equation}
        Finally, we define the family of functions $w^{k}_{\e,\eta}$ as 
        \begin{equation*}
            w^{k}_{\e,\eta}:=\hat w_{\e,\eta(t_{\e,\eta}^k+k)}.
        \end{equation*}
        Not that for every fixed $\eta>0$  $$w^k_{\e,\eta} \rightarrow u \ \quad\text{as}\ \e\to 0\quad  \text{in} \ L^{1}(\Omega;\rr^m)$$  
        uniformly with respect to the choice of $k$. 
        
        We define the set of interactions that cross $\partial\overline A^k_{\varepsilon,\eta}$ as follows
$$
S^{k}_{z,\eta}:=\left\{x \in (A'\cup B)\cap ((A'\cup B)-z): \chi_{\overline A^k_{\varepsilon,\eta}}(x)\neq \chi_{\overline A^k_{\varepsilon,\eta}}(x+z)\right\}.
$$
On the complement of $S^{k}_{z,\eta}$ the function $w^k_{\varepsilon,\eta}$ coincides either with
$u_\varepsilon$ or with $v_\varepsilon$. Since $A'\subset \overline{A}^k_{\e,\eta}\subset A$, the corresponding
energetic contribution is bounded by
$F_\varepsilon(u_\varepsilon,A)+F_\varepsilon(v_\varepsilon,B)$. Hence,
     \begin{equation}
        \label{eq:controllo energia}
            \begin{split}
            F_{\e}(w^{k}_{\e,\eta},A'\cup B)&\leq F_{\e}(u_\e,A)+F_{\e}(v_{\e},B)\\
            &+\int_{\rr^{d}}\int_{S^{k}_{z,\eta}}K_{\e}(x,x+z,w^{k}_{\e,\eta}(x),w^{k}_{\e,\eta}({ x+z}))|D^1_z w^k_{\e,\eta}(x)|\, dx\, dz.
            \end{split}
        \end{equation}

It is straightforward to check that, for every $z\in \rr^d$ such that $|z|<\eta$,
\begin{equation}
    \label{set_inclusion_sr}
    S^{k}_{z,\eta}\subset S^k_{\eta}\cap (S^k_\eta-z).
\end{equation}

Since $\overline{A}^k_{\e,\eta}$ has finite perimeter, we have 
        \begin{equation}
        \label{eq:controllo misura ingrossato del bordo}
            |S^{k}_{z,\eta}|\leq \int_{\rr^d}|\chi_{\overline{A}^k_{\e,\eta}}(x+z)-\chi_{\overline{A}^k_{\e,\eta}}(x)|\, dx\leq |z|\mathcal{H}^{d-1}(\partial \overline{A}^k_{\e,\eta}).
        \end{equation}
        Then,
        \begin{equation*}
            \begin{split}
            \int_{\rr^{d}}\int_{S^{k}_{z,\eta}}&K_{\e}(x,x+z,w^{k}_{\e,\eta}(x),w^{k}_{\e,\eta}(x+z))|D^1_z w^k_{\e,\eta}(x)|\, dx\, dz\\
            &=\int_{|z|\leq \eta}\int_{S^{k}_{z,\eta}}K_{\e}(x,x+z,w^{k}_{\e,\eta}(x),w^{k}_{\e,\eta}(x+z))|D^1_z w^k_{\e,\eta}(x)|\, dx\, dz\\
            &\ \ +\int_{|z|>\eta}\int_{S^{k}_{z,\eta}}K_{\e}(x,x+z,w^{k}_{\e,\eta}(x),w^{k}_{\e,\eta}(x+z))|D^1_z w^k_{\e,\eta}(x)|\, dx\, dz\\
            &=I^{|z|\leq \eta}_\e+I^{|z|> \eta}_\e
            \end{split}
        \end{equation*}
        By (H1)-(H2),  \eqref{eq:diff finite}, \eqref{eq:stima 2}, \eqref{controllo sui moduli dei gradienti}, \eqref{set_inclusion_sr} and the uniform bound  \eqref{eq:nuclei2} we estimate $I^{|z|\leq \eta}_\e$ as follows
        \begin{equation}
          \label{controllo di I_1}
        \begin{split}
            I^{|z|\leq \eta}_\e&\leq \int_{|z|\leq \eta}\int_{S^k_{\eta}\cap (S^k_\eta-z)}K_{\e}(x,x+z,w^{k}_{\e,\eta}(x),w^{k}_{\e,\eta}(x+z))|D^1_z w^k_{\e,\eta}(x)|\, dx\, dz\\
            &\leq \frac{C_2}{C_1}[F_{\e}(u_{\e},S^{k}_{\eta})+F_{\e}(v_{\e},S^{k}_{\eta})]+C_2\int_{|z|\leq \eta}\rho_{\e}(z)\int_{S^{k}_{\eta}}|D^1_{z}\chi_{\overline{A}^k_{\e,\eta}}(x)||u_\e(x)-v_\e(x)|\, dx\, dz\\
            &\leq \frac{C_2}{C_1}[F_{\e}(u_{\e},S^{k}_{\eta})+F_{\e}(v_{\e},S^{k}_{\eta})] + 2\, C_2\int_{|z|\leq \eta}\rho_{\e}(z)\int_{S^{k}_{\eta}}|D^1_{z}\varphi^{k}_{\eta}(x)||u_\e(x)-v_\e(x)|\, dx\, dz\\
            &\leq \frac{C_2}{C_1}[F_{\e}(u_{\e},S^{k}_{\eta})+F_{\e}(v_{\e},S^{k}_{\eta})] +\frac{C}{\eta}\int_{|z|\leq \eta}\rho_{\e}(z)\int_{S^{k}_{\eta}}|u_{\e}(x)-v_{\e}(x)|\, dx\, dz\\
            &\leq\frac{C_2}{C_1}[F_{\e}(u_{\e},S^{k}_{\eta})+F_{\e}(v_{\e},S^{k}_{\eta})] + \frac{C}{\eta}\int_{S^{k}_{\eta}}|u_{\e}(x)-v_{\e}(x)|\, dx.
    \end{split}
        \end{equation}
         On the other hand, by (H2), \eqref{eq:stima 1}, and \eqref{eq:controllo misura ingrossato del bordo} we get
         \begin{equation}
         \label{controllo I_2}
         \begin{split}
             I^{|z|>\eta}_\e&=\int_{|z|>\eta}\int_{S^{k}_{z,\eta}}K_{\e}(x,x+z,w^k_{\e,\eta}(x),w^k_{\e,\eta}(x+z))|D^1_z w^k_{\e,\eta}(x)|\, dx\, dz\\
             &\leq
             C\int_{|z|>\eta}\rho_\e(z)\frac{1}{|z|}\mathcal{H}^{d-1}(\partial \overline{A}^k_{\e,\eta})|z| \, dz\\
             &\leq \frac{C}{\eta }|A^{(k+1)\eta }\setminus A^{k\eta}| \int_{|z|>\eta}\rho_{\e}(z)\, dz,
            \end{split}
         \end{equation}
          where we also used the boundedness of $w_{\e,\eta}^k$.
        Now we note that
\begin{equation}
\label{inclusioni di insiemi}
    \bigcup_{k=0}^{N_{\eta}-1}\left(A^{(k+1)\eta}\setminus A^{k\eta}\right) \subset A\setminus A'
\end{equation}
and 
\begin{equation}
    \label{inclusione nell' intersezione}
    \bigcup_{k=1}^{N_\eta-1}S^{k}_{\eta}\subset A\cap B.
\end{equation}
        Since, for every $x \in A\cap B$ we have $\# \{k: x\in S^{k}_{\eta}\}\leq 3$, by \eqref{eq:controllo energia}, \eqref{controllo di I_1}, \eqref{controllo I_2}, \eqref{inclusioni di insiemi} and \eqref{inclusione nell' intersezione} we get that for $\eta$ sufficiently small
\begin{equation}\label{stimaIepseta}
\begin{aligned}
I_{\varepsilon,\eta}
&:=
\frac{1}{N_\eta-1}
\sum_{k=1}^{N_\eta-1}
F_{\varepsilon}(w^k_{\varepsilon,\eta},A'\cup B)
\\
&\leq
F_{\varepsilon}(u_{\varepsilon},A)
+
F_{\varepsilon}(v_{\varepsilon},B)
+
\frac{3}{N_\eta-1}
\Bigg[
\frac{C_2}{C_1}
\Big(
F_{\varepsilon}(u_{\varepsilon},A\cap B)
+
F_{\varepsilon}(v_{\varepsilon},A\cap B)
\Big)
\\
&\qquad\qquad
+
{\frac{C}{\eta}}
\int_{A\setminus A'}
|v_{\varepsilon}(x)-u_{\varepsilon}(x)|\,dx
+
{\frac{C}{\eta}}
|A\setminus A'|
\int_{|z|>\eta}
\rho_\varepsilon(z)\,dz\Bigg]
\\
&\hspace*{-1em}\leq
F_{\varepsilon}(u_{\varepsilon},A)
+
F_{\varepsilon}(v_{\varepsilon},B)
+
C\eta
+
{C}
\|v_{\varepsilon}-u_{\varepsilon}\|_{L^1(A\setminus A')}
+
{C}
|A\setminus A'|
\int_{|z|>\eta}
\rho_\varepsilon(z)\,dz .
\end{aligned}
\end{equation}
Indeed,
for $\eta$ sufficiently small, we have
$$
N_\eta-1\geq \frac{\overline d}{2\eta},
$$
and the energies $F_{\varepsilon}(u_{\varepsilon},A\cap B)$ and
$F_{\varepsilon}(v_{\varepsilon},A\cap B)$ are uniformly bounded.
Then, there exists
 $k^{\e}_{\eta}\in\{1,\dots,N_{\eta}-1\}$ such that 
\begin{equation*}
    F_{\varepsilon}(w_{\e,\eta}^{k^{\e}_ \eta},A'\cup B)\leq I_{\e,\eta}.
\end{equation*}
Setting $w_{\e,\eta}:=w_{\e,\eta}^{k^{\e}_{\eta}}$, for every fixed $\eta>0$, we have $w_{\varepsilon,\eta}\to u$ in $L^1$ as $\varepsilon\to0$. Then, passing to the limit in \eqref{stimaIepseta} and using the
decay property \eqref{eq:nuclei3}, we obtain
\begin{equation*}
    F''(u,A'\cup B)
    \leq \limsup_{\varepsilon \to 0}
    F_{\varepsilon}(w_{\varepsilon,\eta},A'\cup B)
    \leq F''(u,A)+F''(u,B)+C\eta .
\end{equation*}
Letting $\eta\to0$, the desired conclusion follows.
    \end{proof}

\noindent \textit{Proof of Theorem \ref{main thm}}. 
Taking into account Proposition \ref{prop controllo dall'alto} and \ref{Prop subadditivity}, one can infer that for any $u \in BV(\Omega;Y)$, $F''(u,\cdot)$ is inner regular and subadditive on $\mathcal{A}^{\reg}(\Omega)$ (see as a reference for this procedure \cite[Proposition 11.5 and 11.6]{BD}). By the compactness property of the $\Gamma$-convergence and arguing as in \cite[Theorem 10.3]{BD}, noting that only the inner regularity of $F''$ is needed, for every sequence $\e_j\to 0$ there exists a subsequence $\e_{j_k}\to 0$ such that, for any $(u,A) \in BV(\Omega;Y)\times \mathcal{A}^{\reg}(\Omega)$ there exists 
$$
\Gamma\text{-}\lim_{k\to \infty} F_{\e_{j_{k}}}(u,A)=:F(u,A).
$$ 
Moreover, by setting
\begin{equation}\label{estensione}
    F(u,A):=\sup\{F(u,A'): A' \in \mathcal{A}^{\reg}(\Omega), A'\Subset A\},
\end{equation}
we can extend $F(u,\cdot)$ to $\mathcal{A}(\Omega)$ and verify that all properties satisfied by $F'(u,\cdot)$ and $F''(u,\cdot)$, as stated in Propositions \ref{prop controllo dal basso}, \ref{prop controllo dall'alto} and  \ref{Prop subadditivity} still hold for $F(\cdot,\cdot)$ in \eqref{estensione}. Since in Proposition \ref{prop controllo dal basso} we proved that the domain of the $\Gamma$-limit is $BV(\Omega;Y)$, we have
\begin{equation*}
    \Gamma(L^1)\,\text{-}\lim_{k}F_{\varepsilon_{j_{k}}}(u)=+\infty
\end{equation*}
for $u \in L^{1}(\Omega,Y)\setminus BV(\Omega;Y)$. So far, it suffices to check that, for every $(u,A) \in BV(\Omega;Y)\times \mathcal{A}(\Omega)$, $F(u,A)$ satisfies all the hypotheses of Theorem \ref{representation thm}. Note that $F_{\varepsilon}(u,\cdot)$ is superadditive, and this property is preserved in the limit. Thus, we can apply Theorem \ref{De Giorgi Letta Criterion} and we get hypotheses $(i)$ in Theorem \ref{representation thm}. By Propositions \ref{prop controllo dal basso} and \ref{prop controllo dall'alto} we get $(iv)$, while $(ii)$ is trivially satisfied.
Finally, hypothesis $(iii)$
is a consequence of the lower semicontinuity of the $\Gamma$-limit.
Therefore all the assumptions of Theorem \ref{representation thm} are
satisfied, and the conclusion follows.
\qed

\vspace{5pt}

\section{Convergence of minimum problems.} 
\label{Convergence of minimum problems.}
In this section, we prove the convergence of minimum problems involving energies as  in \eqref{rescale_energy} to the corresponding minimum problem involving the limiting energy. We first consider boundary value problems and subsequently extend the analysis to minimum problems with prescribed volume fraction.

\noindent Our analysis relies on two compactness results proved in \cite{P}, for $d\geq2$ and $d=1$, respectively, which ensure the coerciveness of our family of functionals. The statements of these results are provided below and are slightly rephrased for our purposes. 

\begin{theorem}\cite[Theorem 1.2]{P}
    \label{compactness}
    Assume $d\geq 2$. Let $G_\e(u)$ be defined by \eqref{definizione di G} with $\rho_\e:\rr^d\to [0,+\infty)$ satisfying \eqref{eq:nuclei1}, \eqref{eq:nuclei2}, and \eqref{eq:nuclei3}. If $(u_{\varepsilon})\subset L^{1}(\Omega;\rr^m)$ is a bounded sequence  such that 
    \begin{equation*}
        \sup _\e G_\e(u_\e)<+\infty,
    \end{equation*}
    then $(u_\varepsilon)$ is relatively compact in $L^{1}(\Omega;\rr^m)$. Assume that $u_{\varepsilon_j} \rightarrow u$ in $L^{1}(\Omega;\rr^m)$, then $u \in BV(\Omega;\rr^m)$. 
\end{theorem}

If $d=1$, we need to impose an additional condition on the kernels $\rho_\e$ in order to have compactness. More precisely, given $\rho_\e \in L^1(\rr)$, we shall assume that $\rho_\e$ is defined for every $x \in \rr$ as 
\begin{equation}
\label{def_puntuale_dim_1}
    \rho_\e(x)=
    \begin{cases}
    \displaystyle
        \lim_{r\to 0}\frac{1}{2r}\int_{x-r}^{x+r}\rho_\e(y)\, dy &\text{if } x\  \text{is a Lebesgue point of} \ \rho_\e\\
        +\infty &\text{otherwise.}
    \end{cases}
\end{equation}
Given $\theta_0 \in (0,1)$ and setting 
\begin{equation*}
    \rho_{\e,\theta_0}(x):=\inf_{\theta_0\leq \theta\leq 1} \rho_\e(\theta x)\ \quad \ \forall x \in \rr,
\end{equation*} 
we will consider the following hypothesis:  there exists $\alpha_0 >0$ such that 
\begin{equation}
\label{additional_kernel_assumption}\int_{\rr}\rho_{\e,\theta_0}\geq \alpha _0>0 \ \quad \ \forall \e >0.
\end{equation}

\begin{theorem}\cite[Theorem 1.3]{P}\label{compactness1}
Assume $d=1$ and $\Omega=(a,b)$. Let $G_\e(u)$ be defined by \eqref{definizione di G} with $\rho_\e:\rr\to [0,+\infty)$ defined as in \eqref{def_puntuale_dim_1}
and satisfying \eqref{eq:nuclei1}, \eqref{eq:nuclei2}, and \eqref{eq:nuclei3}. 
Assume there exist $\theta_0 \in (0,1)$ and $\alpha_0>0$ such that \eqref{additional_kernel_assumption} is satisfied. If $(u_\e)\subset L^1((a,b);\rr^m)$ is a bounded sequence such that 
\begin{equation*}
    \sup_\e G_\e(u_\e)<+\infty,
\end{equation*}
then $(u_\e)$ is relatively compact in $L^1((a,b);\rr^m)$; moreover, if $u_{\e_j}\to u$ in $L^1((a,b);\rr^m)$, then $u\in BV((a,b);\rr^m)$. 
\end{theorem}

\subsection{Boundary value minimum problems}
Given $u_0 \in BV(\Omega;Y)$, 
let $A \in \mathcal{A}^{\reg}(\Omega)$ and for any $\eta>0$ set
\begin{equation*}
    A^{\eta}=\{x\in A: \dist(x,\partial A)>\eta\}
\end{equation*}
and 
\begin{equation}
    m^{\eta}_{\e}(u_{0},A):=\inf\{F_{\e}(u,A): u\in L^1(A;Y), u=u_{0}\ \text{on}\ A\setminus A^{\eta}\}.
    \label{minimum problem}
\end{equation}
We observe that $m^{\eta}_{\e}(u_0,A)$ is increasing in $\eta$. 
Moreover $m_\e^{\eta}(u_0,A)$ is uniformly bounded. Indeed,
\begin{equation*}
    m_\e^\eta(u_0,A)\leq F_\e(u_0,A)
\end{equation*}
which is uniformly bounded by  \eqref{GFG}, \eqref{puntuale} and Corollary \ref{liminfG}.

\begin{theorem}
\label{convergence of minimum problem}
    Let $K_{\e}(x,y,u,v)$ satisfy { \rm (H1)-(H2)}. If $d=1$, assume there exist $\theta_0 \in (0,1)$ and $\alpha_0>0$ such that \eqref{additional_kernel_assumption} is satisfied. Given a sequence $\e_{j} \rightarrow 0$, let $(\e_{j_k})$ and $F(u,A)$ be as in Theorem \ref{main thm}. Then it holds 
    \begin{equation*}
        \lim_{\eta \rightarrow 0}\liminf_{k\rightarrow +\infty} m^{\eta}_{\e_{j_k}}(u_{0},A)=\lim_{\eta \rightarrow 0}\limsup_{k\rightarrow +\infty} m^{\eta}_{\e_{j_k}}(u_{0},A)=m(u_{0},A),
    \end{equation*}
    where $m^{\eta}_{\e}(u_{0},A)$ is defined by \eqref{minimum problem} and 
    \begin{equation*}
        m(u_{0},A)=\inf\{F(u,A): u \in BV(A;Y), u=u_{0}\ \text{in a neighborhood of}\ \partial A\}.
    \end{equation*}
    \begin{proof}
        Let $\eta >0$ be fixed and let $u_k \in L^1(A;Y)$ such that  $u_k=u_{0}$ on $A \setminus A^{\eta}$ and
        \begin{equation}
        \label{inf property}
            F_{\e_{j_k}}(u_k,A)\leq m^{\eta}_{\e_{j_k}}(u_{0},A)+\frac{1}{k}.
      \end{equation}
      Since $m_{\e_{j_k}}^\eta(u_0,A)$ is uniformly bounded, by Theorems  \ref{compactness} and \ref{compactness1} there exists a subsequence (not relabeled) such that $u_k$ converges strongly in $L^{1}(A;\rr^m)$ to some $u \in BV(A;Y)$. In particular, $u=u_{0}$ in $A\setminus A^{\eta}$. By Theorem \ref{main thm} and \eqref{inf property} we get
      \begin{equation*}
          m(u_{0},A)\leq F(u,A)\leq \liminf_{k \rightarrow +\infty}F_{\e_{j_k}}(u_{k},A)\leq \liminf_{k \rightarrow +\infty} m^{\eta}_{\e_{j_k}}(u_{0},A),
      \end{equation*}
      and then, taking the limit as $\eta \rightarrow 0$, it follows
      \begin{equation*}
           m(u_{0},A)\leq\lim_{\eta \rightarrow 0} \liminf_{k \rightarrow +\infty} m^{\eta}_{\e_{j_k}}(u_{0},A).
      \end{equation*}
      It remains to prove the reverse inequality. To this end, let $\delta>0$ and choose  $u \in BV(A;Y)$ such that 
      \begin{equation}
      \label{inf property 2}
          F(u,A)\leq m(u_{0},A)+\delta
      \end{equation}
      and $u=u_{0}$ in a neighborhood of $\partial A$. Hence, there exists $\eta_0>0$ such that $u=u_{0}$ in $A \setminus A^{\eta_0}$. 
      Let $u_k \in L^1(A;Y)$ be such that $u_k \rightarrow u$ strongly in $L^{1}(A;\rr^m)$ and 
      \begin{equation*}
          \lim_{k\rightarrow +\infty} F_{\e_{j_k}}(u_k,A)=F(u,A).
      \end{equation*}
       We may argue as in the proof of Proposition \ref{Prop subadditivity} in order to modify the values of $u_k$ in $A \setminus A^{\eta}$ without increasing the energy too much. More precisely, given $\eta\in(0,\eta_0)$, we can select open sets $A_k \in \mathcal{A}^{\reg}(\Omega)$ such that $A^{\eta}\Subset A_{k}\Subset A^{\eta/2}$ and the functions
      \begin{equation*}
          v_{k}=\chi_{A_k}u_k+(1-\chi_{A_k})u_{0}
      \end{equation*}
      satisfy
      \begin{itemize}
          \item[(i)] $v_k=u_{0}$ on $A\setminus A^{\eta/2}$;
          \item[(ii)] $v_{k}\rightarrow u$ strongly in $L^{1}(A;\rr^m)$;
          \item[(iii)] $F_{\e_{j_k}}(v_{k},A)\leq F_{\e_{j_k}}(u_k,A)+ F_{\e_{j_k}}(u_{0}, A \setminus A^{\eta})+C\eta+C\|u_k-u_{0}\|_{L^{1}(A\setminus A^{\eta})}+o(1).$
      \end{itemize}
      By \eqref{GFG}, \eqref{puntuale} and Corollary \ref{liminfG}, we get
      \begin{equation*}
      \begin{split}
          \limsup_{k \rightarrow +\infty }m^{\eta/2}_{\e_{j_k}}(u_{0},A)&\leq \limsup_{k \rightarrow +\infty}F_{\e_{j_k}}(v_{k},A)\\
          &\leq F(u,A)+C\mathcal{H}^{d-1}(S_{u_{0}}\cap (A\setminus A^{\eta}))+C\eta.
        \end{split}
      \end{equation*}
      By \eqref{inf property 2}, we infer that
      \begin{equation*}
          \lim_{\eta \rightarrow 0}\limsup_{k \rightarrow +\infty }m^{\eta}_{\e_{j_k}}(u_{0},A)\leq F(u,A)\leq m(u_{0},A)+\delta
      \end{equation*}
      and the conclusion follows by the arbitrariness of $\delta$.
    \end{proof}
\end{theorem}

\subsection{Minimum problems with prescribed volume fraction}

Given $l=(l^1,l^2,\dots,l^N)\in (0,1)^N$ such that $\sum_{i=1}^{N}l^i=1$, we set 
$$\mathcal{V}^{l}(\Omega;Y):=\{u\in L^1(\Omega;Y):|\{u=y_i\}|=l^i|\Omega|,\   i=1,\dots, N\}$$
and define $F^{l}_\e:L^{1}(\Omega;Y)\rightarrow [0,+\infty]$ as
\begin{equation*}
    F^{l}_\e(u)=
    \begin{cases}
        F_\e(u) &\text{if}\ u \in \mathcal{V}^{l}(\Omega;Y)\\
        +\infty &\text{otherwise,}
    \end{cases}
\end{equation*}
with $F_\e$ as in \eqref{rescale_energy}. The following $\Gamma$-convergence theorem holds.
\begin{theorem}
\label{teo_volume_fratto}
    Let $K_\e(x,y,u,v)$ satisfy {\rm (H1)-(H2)}. If $d=1$ and $\Omega=(a, b)$,  assume in addition that $$K_\e(x,y,u,v)=K_\e(y-x,u,v).$$
    Let $l_\e, l\in  (0,1)^N$ such that   $\sum_{i=1}^{N}l^i_\e=\sum_{i=1}^{N}l^i=1$ and $l_\e \rightarrow l$ as $\e\to 0$. Given $\e_j\rightarrow 0,$ let $\e_{j_k}$ and $F(\cdot)$ as in Theorem \ref{main thm}. Then $F^{l_{\e_{j_k}}}_{\e_{j_k}}$ $\Gamma$-converges with respect to the $L^1(\Omega;\rr^m)$ strong topology to the functional $F^{l}:L^1(\Omega;Y)\rightarrow [0,+\infty]$ defined by
    \begin{equation*}
        F^l(u)=
        \begin{cases}
            F(u) &\text{if}\ |\{u=y_i\}|=l^i|\Omega|, \forall i=1,\dots,N\\
            +\infty &\text{otherwise.}
        \end{cases}
    \end{equation*}
\end{theorem}

\begin{proof}
    For the sake of notation, in what follows we drop the dependence on the sequence $j_k$. Let $u_\e \in \mathcal{V}^{l_\e}(\Omega;Y)$
    converge to $u \in BV(\Omega;Y)$ strongly in $L^1(\Omega;\rr^m)$ such that $$\liminf_{\e\rightarrow 0}F_\e(u_\e)<+\infty.$$
    Then, $u \in \mathcal{V}^l(\Omega;Y)$ and  by Theorem \ref{main thm} the $\Gamma$-liminf inequality follows; i.e.  
    \begin{equation*}
        \liminf_{\e\rightarrow 0} F^{l_\e}_\e(u_\e)\geq F^l(u).
    \end{equation*}
    We now prove the $\Gamma$-limsup inequality. Thanks to the density result stated in \cite[Theorem A.1, Remark A.2]{AG}, it is enough to consider $u \in BV(\Omega;Y)$ such that $|\{u=y_i\}|=l^i|\Omega|$ and whose jump set $S_u$ is polyhedral. Let $u_\e$ be a recovery sequence for $F(u)$. Arguing as in Proposition \ref{Prop subadditivity}, we can find $\delta>0$ and a suitable neighborhood of $S_u$, $A_\e$, such that, for $\e$ sufficiently small, $A_\e \subset S_\delta:=\{x \in \Omega: \dist(x,S_u)\leq \delta\}$ and the sequence of functions \begin{equation*}\tilde{u}_\e:=\chi_{A_\e}u_\e+(1-\chi_{A_\e})u,
    \end{equation*}
    is still a recovery sequence for $F(u)$. In particular $\tilde{u}_\e\equiv u$ in $\Omega\setminus S_\delta$. To get the conclusion, it is enough to replace $\tilde{u}_\e$ with a new sequence $\bar{u}_\e\in \mathcal{V}^{l_\e}(\Omega;Y)$ 
    such that $\bar u_\e\to u$ in $L^1(\Omega;\rr^m)$ and $$F_\e(\bar{u}_\e)\leq F_\e(\tilde{u}_\e)+o(1).$$
   For simplicity, assume $N=3$. Let us first construct $\bar u_\e$ in the case $d\geq 2$. We set $$\tilde{l}^i_\e:=\displaystyle\frac{|\{\tilde{u}_\e=y_i\}|}{|\Omega|},\quad r^i_\e:=\left(\frac{|\tilde{l}^i_\e-l^i_\e||\Omega|}{\omega_d}\right)^{1/d}.$$
 Since $\tilde u_\e\to u$ in $L^1(\Omega;\rr^m)$,
 we have 
\begin{equation}\label{raggipiccoli}
r_\e^i\to0
\quad\text{as }\e\to0\quad \forall i\in\{1,2,3\} .
\end{equation}
Moreover, since $S_u$ is polyhedral and
\begin{equation*}
|\{u=y_i\}|=l^i|\Omega|>0
\qquad\text{for every }i=1,2,3,
\end{equation*}
there exist $\bar r>0$ and $x_i\in \Omega\setminus S_\delta$, $i=1,2,3$, such that
$$
B(x_i,\bar r)\subset \{u=y_i\}\cap (\Omega\setminus S_\delta)=\{\tilde u_\e=y_i\}\cap (\Omega\setminus S_\delta).
$$
By \eqref{raggipiccoli}, we may assume that
$$
\sum_{i=1}^3r_\e^i<\bar r.
$$
 If $\tilde{l}_\e=l_\e$, we simply set $\bar u_\e=\tilde{u}_\e$. Otherwise, there are two components of $l_\e-\tilde l_\e$ that have opposite sign. Assume, for example, that $l_\e^1>\tilde l_\e^1$ and $l_\e^2<\tilde l_\e^2$.  If $l_\e^3>\tilde l_\e^3$, we set
 \begin{equation*}
        \bar u_{\e}(x)=
        \begin{cases}
            \tilde{u}_\e &\text{if}\ x\not\in \tilde B^1_\e\cup \tilde B^3_\e\\
            y_1 &\text{if}\ x \in \tilde B_\e^1,\\
            y_3 &\text{if}\ x \in \tilde B_\e^3,
        \end{cases}
    \end{equation*}
 where $\tilde B^1_\e$ and $\tilde B^3_\e$ are disjoint balls of radius $r_\e^1$ and $r_\e^3$, respectively, and $\tilde B^1_\e\cup \tilde B^3_\e\subset B(x_2,\bar r)$. If $l_\e^3<\tilde l_\e^3$, we set
 \begin{equation*}
        \bar u_{\e}(x)=
        \begin{cases}
            \tilde{u}_\e &\text{if}\ x\not\in B(x_2,r_\e^2)\cup B(x_3,r_\e^3)\\
            y_1 &\text{if}\ x \in B(x_2,r_\e^2)\cup B(x_3,r_\e^3).\\
        \end{cases}
    \end{equation*}
 The case $l_\e^3=\tilde l_\e^3$ is treated analogously by using only one correcting ball. 
In all cases, we have $\bar u_\e\in \mathcal{V}^{l_\e}(\Omega;Y)$ and $\bar u_\e\rightarrow u$ in $L^1(\Omega;\rr^m)$. Moreover, setting $J_\e:= S_{\bar u_\e}\cap (\Omega\setminus S_\delta)$ and $J^z_\e:=\{x=y+tz: y \in  J_\e,|t|\leq1\}\cap \Omega$, by (H2), \eqref{raggipiccoli} and taking into account that, by \eqref{eq:nuclei3}, interactions between points at distance of order $1$ are negligible,  we get
    \begin{equation*}
\begin{split}
        F_{\e}(\bar u_\e)&\leq F_\e(\tilde{u}_\e)+C\int_{\rr^{d}}\frac{\rho_\e(z)}{|z|}|J^z_\e|\, dz \\
        &\leq F_\e(\tilde{u}_\e)+C\mathcal{H}^{d-1}(J_\e)\int_{\rr^d} \rho_\e(z)\,dz\\
        &\le F_\e(\tilde{u}_\e)+C\sum_{i=1}^3(r_\e^i)^{d-1}        =F_\e(\tilde{u}_\e)+o(1),
        \end{split}
    \end{equation*}
    which yields the conclusion.
    The previous construction does not apply in the case $d=1$ since $\mathcal{H}^{d-1}(J_\e)$
    is not negligible as $\e\to 0$. In this case, note that the additional assumption $K_\e(x,y,u,v)=K_\e(y-x,u,v)$ yields that
    \begin{equation}\label{shifting}
       F_\e(u,A)=F_\e(u(\cdot+\tau), A-\tau),\quad \forall \tau\in\rr. 
    \end{equation}
     We then exploit this property and construct $\bar u_\e$ by suitably `locally' shifting $\tilde u_\e$. Let
    $S_u=\{x_1,\dots, x_M\}\subset (a,b).$
    We may assume that $\delta$ is sufficiently small so that $\displaystyle\min_{\substack{h,k\in\{1,\dots,M\}\\ h\neq k}}|x_k-x_h|>4\delta$ and $\displaystyle\min_{k\in\{1,\dots, M\}}|x_k-a|\wedge|x_k-b|>2\delta$. We then set
\begin{equation*}
        \bar u_{\e}(x)=
        \begin{cases}
            \tilde{u}_\e(x)&\text{if}\ x\not\in \displaystyle\bigcup_{k=1}^M (x_k-2\delta,x_k+2\delta) \\
            \tilde{u}_\e(x+\tau_\e^k) &\text{if}\ x \in (x_k-2\delta,x_k+2\delta),\ k\in\{1,\dots, M\},
        \end{cases}
    \end{equation*}
    where $\tau_\e:=(\tau_\e^1,\dots,\tau_\e^M)\to 0$ as $\e\to 0$ and it is suitably chosen so that $\bar u_\e\in \mathcal{V}^{l_\e}(\Omega;Y)$. By \eqref{shifting} and using again that, by \eqref{eq:nuclei3}, interactions between points at distance of order $1$ are negligible,  we get
    $$
    F_\e(\bar u_\e)=\sum_{k=1}^M F_\e(\bar u_\e, (x_k-2\delta,x_k+2\delta))+o(1)=F_\e(\tilde u_\e)+o(1),
    $$
    which yields the conclusion.
    \end{proof}

As a consequence of Theorem \ref{teo_volume_fratto}, the compactness results stated in Theorems \ref{compactness} and \ref{compactness1}, and the main properties of $\Gamma$-convergence, we get the following result.

\begin{theorem}
    Let the hypotheses of Theorem \ref{teo_volume_fratto} be satisfied.  If $d=1$, assume there exist $\theta_0 \in (0,1)$ and $\alpha_0>0$ such that \eqref{additional_kernel_assumption} holds true. Then, it holds
    \begin{equation*}
    \begin{split}
        \lim_{k\rightarrow +\infty}&\min\{F_{\e_{j_k}}(u): |\{u=y_i\}|=l^{i}_{\e_{j_k}}|\Omega|, \ \forall i \in\{1,\dots, N\}\}\\
        =&\min\{F(u): |\{u=y_i\}|=l^i|\Omega|, \ \forall i \in\{1,\dots, N\}\}. 
        \end{split}
    \end{equation*}
    Moreover, if $u_{\e_{j_k}}$ is a converging sequence such that $|\{u_{\e_{j_k}}=y_i\}|=l^i_{\e_{j_k}}|\Omega|$ for every $i\in\{1,\dots,N\}$ and
    \begin{equation*}
        \lim_{k\rightarrow +\infty} F_{\e_{j_k}}(u_{\e_{j_k}})=\lim_{k\rightarrow +\infty}\min\{F_{\e_{j_k}}(u): |\{u=y_i\}|=l^i_{\e_{j_k}}|\Omega|, \ \forall i \in\{1,\dots, N\}\},
    \end{equation*}
    then its limit is a minimizer for $\min\{F(u): |\{u=y_i\}|=l^i|\Omega|, \ \forall i \in\{1,\dots, N\}\}$.
\end{theorem}

\section{Homogenization.}
\label{Homogenization}
In this section, we study the homogenization of the functionals $F_\e$ defined in \eqref{rescale_energy} under a periodicity assumption in the space variable of the kernels $K_\e$.
Assume that $K_{\e}(x,y,u,v)=\displaystyle\frac{1}{\e^d}K\left(\frac{x}{\e},\frac{y}{\e},u,v\right)$ , where $K(\cdot, \cdot,u,v):\rr^d\times \rr^d\rightarrow [0,+\infty)$ is $1$-periodic with respect to simultaneous translations of the two
spatial variables, namely
\begin{equation}
\label{periodicita}
    K(x+h,y+ h,u,v)=K(x,y,u,v)\ \quad \ \text{for every}\ x,y\in \rr^d, \ u,v\in Y,\ \text{and}\ h \in \mathbb{Z}^d .
\end{equation}
The analog of assumptions (H1) and (H2) is given by the following set of hypotheses on the function $K$: 
there exists a radial Borel function $\rho:\rr^d\to [0,+\infty)$ such that 
\begin{equation}
    \label{condizione_rho_hom}\int_{\rr^d}\rho(z)\, dz=1
\end{equation}
and for a.e. $(x,y) \in \rr^d\times \rr^d$ and for every $(u,v)\in Y \times Y$ it holds:
\begin{itemize}
    \item[(H1')]\textit {(Coerciveness)} $\exists \,C_1>0$ such that
    $$ K(x,y,u,v)\geq C_1\rho(y-x)$$
 \item[(H2')] \textit{(Growth)} $ \exists\, C_2>0$ such that $$K(x,y,u,v)\leq C_2 \rho(y-x).$$   
\end{itemize}
Setting 
\begin{equation*}
\rho_\e(z)=\displaystyle \frac{1}{\e^d}\rho\left(\frac{z}{\e}\right),
\end{equation*}
it is straightforward to verify that if  (H1') and (H2') hold, then $K_\e$ satisfies assumptions (H1) and (H2). 

 Throughout the section we will use the change of variable $z=y-x$. For any open set $A\subset \rr^d$ and $z \in \rr^d$ we set
\begin{equation*}
    A_z=A\cap (A-z)=\{x \in A: x+z \in A\}.
\end{equation*}
Then, the localized energy in \eqref{rescale_energy} can be written as
\begin{equation}
\label{energy_hom_kernel}
    F_\e(u,A)=\frac{1}{\e^d}\int_{\rr^d}\int_{A_z}K\left(\frac{x}{\e},\frac{x+z}{\e},u(x),u(x+z)\right)\frac{|u(x+z)-u(x)|}{|z|}\, dx\, dz.
\end{equation}

\begin{remark}
\label{rmk_hyp_hom}
Set
\begin{equation*}
    \widehat{K}\left(x,z,u,v\right):= K\left(x,x+z,u,v\right )\ \quad \ \widehat{K}_\e\left(x,z,u,v\right):= K_\e\left(x,x+z,u,v\right ).
\end{equation*}
In particular,
\begin{equation*}
    \widehat{K}_\e(x,z,u,v)=\frac{1}{\e^d}\widehat{K}\left(\frac{x}{\e},\frac{z}{\e},u,v\right).
\end{equation*}
By the periodicity
assumption on $K$, the function $\widehat K$ is $1$-periodic with
respect to the first variable and then $\widehat{K}_\e$ satisfies
\begin{equation}
\label{periodicity_first_variable}
    \widehat{K}_\e(x+\e h,z,u,v)=\widehat{K}_\e(x,z,u,v),\ \quad \ \text{for every $h \in \mathbb{Z}^d$.}
\end{equation}
Moreover, by (H1') and (H2'),
\begin{equation}
   \label{cv_crescita} C_1\rho(z)\leq \widehat{K}\left(x,z,u,v\right)\leq C_2\rho(z),\ \quad \  C_1\rho_\e(z)\leq \widehat{K}_\e\left(x,z,u,v\right)\leq C_2\rho_\e(z).
\end{equation}
With this notation, \eqref{energy_hom_kernel} can be written as
\begin{equation*}  F_\e(u,A)=\int_{\rr^d}\int_{A_z}\widehat{K}_\e(x,z,u(x),u(x+z))\frac{|u(x+z)-u(x)|}{|z|}\, dx\, dz.
\end{equation*}
\end{remark}

The following theorem holds.
\begin{theorem}
\label{teo:homogenization}
Let $F_\e$ be defined by \eqref{rescale_energy} and let $K_{\e}(x,y,u,v)=\displaystyle\frac{1}{\e^d}K\left(\frac{x}{\e},\frac{y}{\e},u,v\right)$, where $K(\cdot,\cdot,u,v)$ satisfies \eqref{periodicita}. If {\rm (H1')} and {\rm (H2')} hold, then the functionals $F_\e$ $\Gamma$-converge with respect to the $L^1(\Omega;\rr^m)$ strong topology, as $\e \rightarrow 0^+,$ to the homogenized functional $F_{\hom}:L^1(\Omega;Y)\rightarrow [0,+\infty]$ defined by
\begin{equation*}
    F_{\hom}(u)=
    \begin{cases}
        \displaystyle\int_{S_u}\varphi_{\hom}(u^+,u^-,\nu_u)\, d\mathcal{H}^{d-1} &\text{if}\ u\in BV(\Omega;Y)\\
            +\infty \ &\text{otherwise}. 
    \end{cases}
\end{equation*}
where 
\begin{equation}
\label{densita hom}
\begin{split}
    \varphi_{\hom}(a,b,\nu) = \inf_{ 0<\eta<1} \limsup_{T \rightarrow +\infty} \frac{1}{T^{d-1}} 
    \inf \left\{ 
        F_1(u, TQ_\nu) : 
        \begin{array}{l}
            u \in L^1(TQ_\nu; Y), \\
            u = u_{a,b,\nu} \ \text{on} \ TQ_\nu \setminus (1-\eta)TQ_\nu 
        \end{array}
    \right\},  
\end{split}
\end{equation}
with $u_{a,b,\nu}$ defined as in \eqref{eq: def u}.
\begin{proof}
    By Theorem \ref{main thm}, given a sequence $\e_j\rightarrow 0^+$, we can  extract a subsequence $\e_{j_k}$ such that the functionals $F_{\e_{j_k}}$ $\Gamma$-converge to a functional $F$ defined as in \eqref{gamma limite} and for any $u \in BV(\Omega;Y)$ and $A \in \mathcal{A}^{\reg}(\Omega)$ 
    \begin{equation*}
        \Gamma(L^1)\,\text{-}\lim_{k} F_{\e_{j_k}}(u,A)=\int_{S_u\cap A}\varphi(x,u^+,u^-,\nu_u)\, d\mathcal{H}^{d-1}.
    \end{equation*}
     For simplicity we write $\e$ instead of $\e_{j_k}$, and all the limits as $\e\to 0$ are understood along this subsequence.
    In order to prove Theorem \ref{teo:homogenization} we first show that $\varphi$ does not depend on $x$. More precisely, for every fixed $x_0,z_0 \in \Omega,\ a,b \in Y$ and $\nu \in S^{d-1}$, we have 
    \begin{equation*}\varphi(x_0,a,b,\nu)=\varphi(z_0,a,b,\nu).
    \end{equation*}
By \eqref{densita}, to get the claim it is enough to prove that, for every $r>0$ such that $Q_\nu(x_0,r)\cup Q_{\nu}(z_0,r)\subset \Omega$, it holds
\begin{equation*}
    m(u_{x_0,a,b,\nu},Q_\nu(x_0,r))=m(u_{z_0,a,b,\nu},Q_\nu(z_0,r)).
\end{equation*}
Given $\delta>0$, let $w \in BV(Q_\nu(x_0,r); Y)$ such that $w=u_{x_0,a,b,\nu}$ in $Q_{\nu}(x_0,r)\setminus Q_\nu(x_0,(1-\eta)r)$ for some $\eta\in (0,1)$ and
\begin{equation}
    \label{quasi_minimizer}
    F(w,Q_\nu(x_0,r))\leq m(u_{x_0,a,b,\nu},Q_\nu(x_0,r))+\delta.
\end{equation}
For $x\in Q_\nu(z_0,r)$, we consider the function $\widetilde{w}(x):=w(x+x_0-z_0)$. We note that $\widetilde{w}=u_{z_0,a,b,\nu}$ in $Q_\nu(z_0,r)\setminus Q_\nu(z_0,(1-\eta)r)$. Let $u_\e$ be a recovery sequence for $F(w,Q_\nu(x_0,r))$. We define $\tau_\e \in \rr^d$ such that, for each $i=1,\dots,d$ the $i$-th component of $\tau_\e$ is given by 
\begin{equation*}  \tau^i_\e=\bigg\lfloor\frac{x^i_0-z_0^i}{\e}\bigg\rfloor \e
\end{equation*}
and, for $x\in Q_\nu(z_0,r)$, $\widetilde{u}_\e(x):=u_\e(x+\tau_\e)$. Let $r'<r$. For $\e$ sufficiently small $ Q_\nu(z_0,r')+\tau_\e\subset Q_\nu(x_0,r)$. Since $u_\e\to  w$ in $L^1(Q_\nu(x_0,r))$, it holds that $\widetilde{u}_\e\to \widetilde{w}$ in $L^{1}(Q_\nu(z_0,r'))$.
 Since $(Q_{\nu}(z_0,r'))_z+\tau_\e=(Q_{\nu}(z_0,r')+\tau_\e)_z$, by the change of variables $x'=x+\tau_\e$ and \eqref{periodicity_first_variable} we obtain
\begin{equation*}
    \begin{split}
        F_\e&(\widetilde{u}_\e,Q_\nu(z_0,r'))=\int_{\rr^d}\int_{(Q_\nu(z_0,r'))_z}\widehat{K}_\e\Bigl(x,z,\widetilde{u}_\e(x),\widetilde{u}_\e(x+z) \Bigr)\frac{|\widetilde{u}_\e(x+z) -\widetilde{u}_\e(x)|}{|z|}\, dx\, dz\\
        &=\int_{\rr^d}\int_{(Q_\nu(z_0,r')+\tau_\e)_z}\widehat{K}_\e\Bigl(x'-\tau_\e,z,u_\e(x'),u_\e(x'+z) \Bigr)\frac{|u_\e(x'+z) -u_\e(x')|}{|z|}\, dx'\, dz\\
        &= F_\e(u_\e, Q_\nu(z_0,r')+\tau_\e)\leq F_\e(u_\e, Q_\nu(x_0,r))
    \end{split}
\end{equation*}
Then,
\begin{equation*}
\begin{split}
    F(\widetilde{w},Q_{\nu}(z_0,r'))&\leq \liminf_{\e \to 0} F_\e(\widetilde{u}_\e,Q_\nu(z_0,r'))\le\lim_{\e\to 0} F_\e(u_\e,Q_\nu(x_0,r))=F(w,Q_\nu(x_0,r)).
\end{split}
\end{equation*}
Therefore, using the inner regularity 
\begin{equation}
\label{equality_limit_functional}
    F(\widetilde{w},Q_{\nu}(z_0,r))=\sup_{r'<r} F(\widetilde{w},Q_\nu(z_0,r'))\leq F(w,Q_\nu(x_0,r)).
\end{equation}
 Since $\widetilde{w}$ is an admissible function for the minimum problem $m(u_{z_0,a,b,\nu},Q_\nu(z_0,r))$, using \eqref{quasi_minimizer} and \eqref{equality_limit_functional}, the following chain of inequalities holds
\begin{equation*}
    m(u_{z_0,a,b,\nu},Q_\nu(z_0,r))\leq F(\widetilde{w},Q_{\nu}(z_0,r))\leq F(w,Q_\nu(x_0,r))\leq m(u_{x_0,a,b,\nu},Q_\nu(x_0,r))+\delta.
\end{equation*}
By the arbitrariness of $\delta$, it follows
\begin{equation*}
    m(u_{z_0,a,b,\nu},Q_\nu(z_0,r))\leq m(u_{x_0,a,b,\nu},Q_\nu(x_0,r)).
\end{equation*}
Reversing the role of $x_0$ and $z_0$ we get the equality. 
We now prove that $\varphi$ coincides with the density in \eqref{densita hom}. Let $m^{\eta}_\e(u_{a,b,\nu},A)$ be defined as in \eqref{minimum problem}. 
    Observe that, if $A=Q$ is a $d$-cube, Theorem \ref{convergence of minimum problem} is still valid if we assume that the test functions in the definitions of $m^{\eta}_\e(u_{a,b,\nu},Q)$ 
    satisfy $u=u_{a,b,\nu}$ on $Q\setminus (1-\eta)Q$.
    By Theorem \ref{convergence of minimum problem}
    and rescaling the functionals, we obtain
    \begin{equation*}
        \begin{split}
            \varphi(a,b,\nu)&=\limsup_{r\rightarrow0^{+}}\frac{1}{r^{d-1}}m(u_{a,b,\nu},r Q_{\nu})\\
            &=\limsup_{r\rightarrow0^{+}}\frac{1}{r^{d-1}}\lim_{\eta \rightarrow 0}\limsup_{k\rightarrow +\infty}m^{\eta}_{\e_{j_k}}(u_{a,b,\nu},r Q_{\nu})\\
            &=\limsup_{r\rightarrow0^{+}}\frac{1}{r^{d-1}}\inf_{ 0<\eta<1}\limsup_{k\rightarrow +\infty}\inf \left\{ 
        F_{\e_{j_k}}(u, r Q_\nu) : 
        \begin{array}{l}
            u \in L^1(r Q_\nu; Y), \\
            u = u_{a,b,\nu} \ \text{on} \ r Q_\nu \setminus (1-\eta)r Q_\nu
        \end{array}
    \right\}\\
    &=\limsup_{r\to 0^+}\inf_{0<\eta<1}\limsup_{k \rightarrow +\infty}\frac{1}{T_{k}^{d-1}}\inf \left\{ 
        F_1(u, T_k Q_\nu) : 
        \begin{array}{l}
            u \in L^1(T_k Q_\nu; Y), \\
            u = u_{a,b,\nu} \ \text{on} \ T_k Q_\nu \setminus (1-\eta)T_k Q_\nu
        \end{array}
    \right\}
        \end{split}
    \end{equation*}
    where  $T_k=\displaystyle \frac{r}{\e_{j_k}}$.
    For any fixed $T>0$ we set
    \begin{equation*}
        m(T,\eta):=\inf\{F_1(u,TQ_\nu): u \in L^1(TQ_{\nu};Y), u=u_{a,b,\nu}\ \text{on}\ T Q_\nu \setminus (1-\eta)TQ_\nu\},
    \end{equation*}
    \begin{equation*}
        m'(\eta):=\liminf_{T \rightarrow +\infty}\frac{1}{T^{d-1}}m(T,\eta), \ \quad  m''(\eta):=\limsup_{T \rightarrow +\infty}\frac{1}{T^{d-1}}m(T,\eta).
    \end{equation*}
    Since 
    \begin{equation*}
        m'(\eta)\leq \liminf_{T_k\rightarrow +\infty}\frac{1}{T_k^{d-1}}m(T_k,\eta)\leq \limsup_{T_k\rightarrow +\infty}\frac{1}{T_k^{d-1}}m(T_k,\eta)\leq m''(\eta),
    \end{equation*}
    to get the conclusion it is enough to prove that
    \begin{equation}
        \label{m'=m''}
        \inf_{0<\eta<1}m'(\eta)=\inf_{ 0<\eta<1}m''(\eta).
    \end{equation}
    The equality \eqref{m'=m''} follows once we prove that, for every $\eta\in(0,1)$,
    \begin{equation}
    \label{liminf_controlla_limsup}
        m''(\eta)\leq m'(\eta)+C\eta.
    \end{equation} 
 In what follows we consider the case $d\ge 2$. The proof can be easily adapted to the case $d=1$ and it is left to the reader. 
 
 Let $B_{\nu}:=\{\nu^{1},\dots,\nu^{d}\}$ be an orthonormal basis of $\rr^d$ such that $\nu^{d}=\nu$ and 
    \begin{equation*}
        Q_{\nu}=\bigoplus^{d}_{j=1}\left(-\frac{1}{2},\frac{1}{2}\right)\nu^{j}.
    \end{equation*}
Let $S\gg T$ and fix $\delta>0$. Let $\bar u_T \in L^1(TQ_\nu;Y)$ such that 
\begin{equation*}
    \bar u_T=u_{a,b,\nu}\ \quad \text{on}\ TQ_\nu\setminus (1-\eta)T Q_\nu
\end{equation*}
and
\begin{equation}
    \label{energy_quasi_minimizer_small_cell}
    F_1(\bar u_T,TQ_\nu)\leq m(T,\eta)+\delta T^{d-1}.
\end{equation}
We introduce the auxiliary length
\begin{equation*}
    \ell_{T,\eta}
:=
T+\frac{\eta T}{2}+2\sqrt d
\end{equation*}
and let $\Pi^{\nu}:=\{x \in \rr^d: \langle x,\nu \rangle=0\}$ be the hyperplane orthogonal to $\nu$. Consider the lattice
\begin{equation*}
    \Lambda_{T,\eta}:=\Bigl\{z_k=\ell_{T,\eta}\sum_{j=1}^{d-1}k_j\nu^j: k=(k_1,\dots,k_{d-1}) \in \mathbb{Z}^{d-1}\Bigr\}\subset \Pi^{\nu}.
\end{equation*}
For every $k \in \mathbb{Z}^{d-1}$, choose $y_k \in \mathbb{Z}^{d}$ such that 
\begin{equation}
    \label{integer_approx_center}
    |y_k-z_k|\leq \sqrt{d}.
\end{equation}
In particular, since $z_k \in \Pi^\nu$, 
\begin{equation}
    \label{normal_displacement_centers}
    |\langle y_k,\nu\rangle|=|\langle y_k-z_k,\nu \rangle|\leq \sqrt{d}.
\end{equation}
We set 
\begin{equation*}
    I_{S,T}:=\Bigl\{k \in \mathbb{Z}^{d-1}:|k_j|\ell_{T,\eta}\leq \frac{(1-\eta)S-T-2\sqrt{d}}{2}\ \text{for every} \ j=1,\dots,d-1\Bigr\}
\end{equation*}
and
\begin{equation*}
    N_{S,T}:=\# I_{S,T}.
\end{equation*}
For every $k \in I_{S,T}$, we define
\begin{equation*}
    Q_k:=y_k+TQ_\nu.
\end{equation*}

\noindent By \eqref{integer_approx_center}, \eqref{normal_displacement_centers}, and the definition of $I_{S,T}$, we have
\begin{equation}
    \label{small_cubes_contained}
    Q_k\subset (1-\eta)SQ_\nu\cap\left\{x\in \rr^d: |\langle x,\nu\rangle|\leq\frac T2+\sqrt{d} \right\}\ \quad\  \forall\ k\in I_{S,T}.
\end{equation}
Moreover, the cubes $Q_k$ are pairwise disjoint. Indeed, for every $k,h \in I_{S,T}$ such that $k\neq h$, there exists $j \in \{1,\dots,d-1\}$ such that $|k_j-h_j|\geq 1$. Hence,
\begin{equation*}
\begin{split}
    |\langle y_k-y_h, \nu^j\rangle|&=|\langle z_k-z_h,\nu^j\rangle+\langle y_k-z_k,\nu^j\rangle-\langle y_h-z_h,\nu^j\rangle|\\
    &\geq |\langle z_k-z_h,\nu^j\rangle|-|\langle y_k-z_k,\nu^j\rangle|-|\langle y_h-z_h,\nu^j\rangle|\\
    &\geq \ell_{T,\eta}-2\sqrt{d}\\
    &=T+\frac{\eta T}{2},
\end{split}
\end{equation*}
 from which it follows that
\begin{equation}
    \label{separation_small_cubes}
    \dist(Q_k,Q_h)\geq \frac{\eta T}{2}, \ \quad \text{for every}\  k \neq h.
\end{equation}
Arguing analogously, we infer that 
\begin{equation}\label{mindist}
\dist\Bigg(Q_k,\bigcup_{\substack{h\in I_{S,T}\\
    h\neq k}}Q_h\Bigg)\leq \frac{\eta T}{2}+4\sqrt{d}\leq\eta T
\end{equation}
for $T$ large enough.
By the definition of $I_{S,T}$, we also have
\begin{equation*}
    N_{S,T}=\Bigl(2\Big\lfloor\frac{(1-\eta)S-T-2\sqrt{d}}{2\ell_{T,\eta}}\Big\rfloor+1\Bigr)^{d-1}.
\end{equation*}
Consequently,
\begin{equation*}
\lim_{S\to+\infty}
\frac{N_{S,T}}{S^{d-1}}
=
\left(
\frac{1-\eta}{\ell_{T,\eta}}
\right)^{d-1}
\end{equation*}
and therefore, since $\ell_{T,\eta} \geq T$,
\begin{equation}
\label{asymptotic_covered_fraction}
\lim_{S\to+\infty}
\frac{N_{S,T}T^{d-1}}{S^{d-1}}
=
(1-\eta)^{d-1}
\left(
\frac{T}{\ell_{T,\eta}}
\right)^{d-1}
\leq1;
\end{equation}
i.e.
\begin{equation}
\label{stima_cardinalita}
    N_{S,T}\leq C\frac{S^{d-1}}{T^{d-1}}.
\end{equation}
We now define $u_S\in L^1(SQ_\nu;Y)$ by
\begin{equation*}
u_S(x)
:=
\begin{cases}
\bar u_T(x-y_k)
&\text{if }x\in Q_k\ 
\text{for}\ k\in I_{S,T},
\\[1mm]
u_{a,b,\nu}(x)
&\text{otherwise}.
\end{cases}
\end{equation*}
Since the cubes $Q_k$
 are pairwise disjoint, $u_S$ is well defined. Moreover, by \eqref{small_cubes_contained}, the function $u_S$ coincides with
$u_{a,b,\nu}$ on
\begin{equation*}
SQ_\nu\setminus(1-\eta)SQ_\nu.
\end{equation*}
Hence $u_S$ is admissible for $m(S,\eta)$, and thus
\begin{equation}
\label{mS_upper_bound_uS}
m(S,\eta)\leq F_1(u_S,SQ_\nu).
\end{equation}
We decompose $F_1(u_S,SQ_\nu)$ as 
\begin{equation*}
F_1(u_S,SQ_\nu)
=
F^{\rm cell}_{S,T}
+
F^{\rm rem}_{S,T},
\end{equation*}
where 
\begin{equation*}
    F^{\rm cell}_{S,T}=\sum_{k \in I_{S,T}}F_1(u_S,Q_k)
\end{equation*}
denotes the contribution of the cubes $Q_k$ and $F^{\rm rem}_{S,T}$ the remaining contribution.
By the periodicity of $\widehat K_1$ in its first
variable (see \eqref{periodicity_first_variable} for $\e=1$) and the fact that $y_k\in\mathbb Z^d$, for every
$k\in I_{S,T}$ we have
\begin{equation*}
F_1(u_S,Q_k)=
F_1(\bar u_T,TQ_\nu).
\end{equation*}
Therefore, by \eqref{energy_quasi_minimizer_small_cell}, we get 
\begin{equation}
\label{estimate_cell_part}
F^{\rm cell}_{S,T}
\leq
N_{S,T}
\left(
m(T,\eta)+\delta T^{d-1}
\right).
\end{equation}
We now estimate $F_{S,T}^{\rm rem}$ by distinguishing interactions of
length at most $\eta T/4$ from the remaining ones. Namely,
\begin{equation*}
F_{S,T}^{\rm rem}
=
F^{sr}_{S,T}
+
F^{lr}_{S,T},
\end{equation*}
where
\begin{equation*}
F^{sr}_{S,T}
:= F_1^{\frac{\eta T}{4}}(u_S,SQ_\nu)-\sum_{k \in I_{S,T}}F_1^{\frac{\eta T}{4}}(u_S,Q_k)
\end{equation*}
and
\begin{equation*}
\begin{split}
F^{lr}_{S,T}
:=& (F_1(u_S,SQ_\nu)-F_1^{\frac{\eta T}{4}}(u_S,SQ_\nu))-\sum_{k \in I_{S,T}}(F_1(u_S,Q_k)-F_1^{\frac{\eta T}{4}}(u_S,Q_k)).
\end{split}
\end{equation*}
We estimate $F^{sr}_{S,T}$.  Set
\begin{equation*}
U_{S,T}:=
\bigcup_{k\in I_{S,T}}Q_k
\end{equation*} 
 and let us first estimate the interactions between points in 
\begin{equation*}
    U_{S,T}^{\eta}:={\{x\in \rr^d: \dist(x,U_{S,T})\leq \eta T\}}.
\end{equation*}
Suppose that $x,x+z \in U_{S,T}^\eta$.
We recall that, by \eqref{separation_small_cubes}, if $|z|\leq \eta T/4$, the points $x$ and
$x+z$ cannot belong to two distinct cubes. Hence, the points $x$ and $x+z$ whose interactions are not taken into account in $F^{\rm cell}_{S,T}$ are such that both points lie outside
$U_{S,T}$ or one of them lies in $\bigcup_{k \in I_{S,T}}\left(Q_k\setminus(y_k+(1-\eta)TQ_\nu)\right)$. In the first case,   
\begin{equation*}
u_S(x)=u_{a,b,\nu}(x),
\qquad
u_S(x+z)=u_{a,b,\nu}(x+z).
\end{equation*}
If these two values are different, the segment
$[x,x+z]$ intersects $\Pi^\nu$.
By \eqref{mindist} and \eqref{stima_cardinalita}, the energetic contribution of these interactions is bounded by
\begin{equation}
\label{aux_estimate_short_rem_new}
    C\left(\frac{S}{T}\right)^{d-1}T^{d-2}T\eta\int_{|z|\leq \frac{\eta T}{4}}\rho(z)\, dz\leq C S^{d-1}\eta.
\end{equation}
Suppose now that $x\in Q_k\setminus (y_k+(1-\eta)TQ_\nu)$ and
$x+z\notin U_{S,T}$. Then
\begin{equation*}
u_S(x)=u_{a,b,\nu}(x-y_k),
\qquad
u_S(x+z)=u_{a,b,\nu}(x+z).
\end{equation*}
If these two values are different, then
\begin{equation*}
\langle x-y_k,\nu\rangle
\quad\text{and}\quad
\langle x+z,\nu\rangle
\end{equation*}
have opposite signs, from which we easily infer that
\begin{equation}\label{stimanormale}
|\langle x,\nu\rangle|\vee |\langle x-y_k,\nu\rangle|\leq|\langle y_k,\nu\rangle|+|\langle z,\nu\rangle|
\leq \sqrt d+|z|.
\end{equation}
Moreover, for $T$ sufficiently large, the segment $[x,x+z]$
cannot leave $Q_k$ through one of the two faces orthogonal to $\nu$.
Indeed, this would require
\begin{equation*}
|\langle x-y_k,\nu\rangle|
\geq \frac{T}{2}-|z|,
\end{equation*}
which is incompatible with (\ref{stimanormale}) for $T$ sufficiently large,
since $|z|\leq \eta T/4$. Therefore, the segment must intersect one of
the lateral faces of $Q_k$. Taking into account also (\ref{stimanormale}), we infer  that  there exists $\overline j\in\{1,\dots, d-1\}$ such that, setting
$$
D_k:=\left\{\xi\in \overline Q_k:\ \langle \xi-y_k,\nu^{\overline j}\rangle=\pm \frac T2,\ |\langle \xi,\nu\rangle|\leq \sqrt{d}+2|z|\right\},
$$
we have
$$
[x,x+z]\cap D_k\neq\emptyset.
$$
Hence,
$$
x\in\{x=x'+tz,\ x'\in D_k,\ |t|\leq 1\}
$$
whose measure is bounded by
\begin{equation*}
CT^{d-2}(\sqrt d+|z|)|z|.
\end{equation*}
By \eqref{stima_cardinalita}
and since $|z|\leq \eta T/4$, the energetic contribution of these interactions is bounded by
\begin{equation}\label{stimareminout}
\begin{split}
C\left(\frac{S}{T}\right)^{d-1}T^{d-2}
\int_{|z|\leq\frac{\eta T}{4}}
\rho(z)(\sqrt d+|z|)\,dz
 \leq
CS^{d-1}
\left(
\frac{1}{T}+\eta
\right).
\end{split}
\end{equation}

The remaining interactions in $F^{sr}_{S,T}$, which  have not been taken into account so far, are those between points $x$ and $x+z$ such that the segment $[x,x+z]$ intersects $\Pi^\nu$ at distance of order $S\eta$ from $\partial (SQ_\nu)$. Arguing as above the energetic contribution of these interactions is bounded by
\begin{equation}
\label{bound_energy_contribution_left}
    CS^{d-1}\eta\int_{|z|\leq \frac{\eta T}{4}}\rho(z)\, dz\leq CS^{d-1}\eta.
\end{equation}
Hence, combining \eqref{aux_estimate_short_rem_new} and \eqref{stimareminout}, \eqref{bound_energy_contribution_left}, we get
\begin{equation}
\label{estimate_short_remainder_new}
    F^{sr}_{S,T}\leq CS^{d-1}\eta.
\end{equation}

We now estimate $F^{lr}_{S,T}$.
For every $z\in\mathbb R^d$,
\begin{equation*}
\left|
\left\{
x\in(SQ_\nu)_z:
x\in U_{S,T}\text{ or }x+z\in U_{S,T}
\right\}
\right|
\leq2|U_{S,T}|
=
2N_{S,T}T^d.
\end{equation*}
Therefore, by \eqref{cv_crescita}, the energetic contribution of the
pairs of points such that one of them is in $U_{S,T}$ is bounded by
\begin{equation}
\label{long-interaction-inside-union}
CN_{S,T}T^d
\int_{|z|>\frac{\eta T}{4}}\frac{\rho(z)}{|z|}\,dz.
\end{equation}
If both points are outside $U_{S,T}$,  then
\begin{equation*}
u_S(x)=u_{a,b,\nu}(x),
\qquad
u_S(x+z)=u_{a,b,\nu}(x+z).
\end{equation*}For every $z\in \rr^d$,
\begin{equation}
\label{flat-jump-measure-estimate}
\left|
\left\{
x\in(SQ_\nu)_z:
u_{a,b,\nu}(x)
\neq
u_{a,b,\nu}(x+z)
\right\}
\right|
\leq
CS^{d-1}|z|.
\end{equation}
Consequently, combining \eqref{long-interaction-inside-union} and \eqref{flat-jump-measure-estimate}
\begin{equation*}
\begin{split}
F^{lr}_{S,T}
\leq{}
CN_{S,T}T^d
\int_{|z|>\frac{\eta T}{4}}\frac{\rho(z)}{|z|}\,dz
+
CS^{d-1}
\int_{|z|>\frac{\eta T}{4}}
\rho(z)\,dz .
\end{split}
\end{equation*}
Since
\begin{equation*}
\int_{|z|>\frac{\eta T}{4}}\frac{\rho(z)}{|z|}\,dz
\leq
\frac{4}{\eta T}
\int_{|z|>\frac{\eta T}{4}}\rho(z)\,dz
\end{equation*}
and
\begin{equation*}
N_{S,T}
\leq
C\left(\frac ST\right)^{d-1},
\end{equation*}
 we obtain 
\begin{equation}
\label{estimate-long-interactions}
F^{lr}_{S,T}
\leq
\frac{CS^{d-1}}{\eta}
\int_{|z|>\frac{\eta T}{4}}\rho(z)\,dz+CS^{d-1}\int_{|z|>\frac{\eta T}{4}}\rho(z)\, dz.
\end{equation}
Combining \eqref{mS_upper_bound_uS}, \eqref{estimate_cell_part},
\eqref{estimate_short_remainder_new}, and
\eqref{estimate-long-interactions}, we find
\begin{equation*}
\begin{split}
m(S,\eta)
&\leq
N_{S,T}
\left(
m(T,\eta)+\delta T^{d-1}
\right)+CS^{d-1}\eta
\\
&+
\frac{CS^{d-1}}{\eta}
\int_{|z|>\frac{\eta T}{4}}\rho(z)\,dz
+CS^{d-1}\int_{|z|>\frac{\eta T}{4}}\rho(z)\,dz.
\end{split}
\end{equation*}
Dividing by $S^{d-1}$, we obtain
\begin{equation*}
\begin{split}
\frac{1}{S^{d-1}}m(S,\eta)
\leq{}&\frac{N_{S,T}T^{d-1}}{S^{d-1}}
\left(\frac{1}{T^{d-1}}m(T,\eta)
+\delta\right)
+C\eta
\\
&+
\frac C\eta
\int_{|z|>\frac{\eta T}{4}}\rho(z)\,dz+C\int_{|z|>\frac{\eta T}{4}}\rho(z)\,dz.
\end{split}
\end{equation*}
Letting first $S$ and then $T$ tend to $+\infty $, by \eqref{condizione_rho_hom} and
\eqref{asymptotic_covered_fraction} we get
\begin{equation*}
\limsup_{S\to +\infty}\frac{1}{S^{d-1}}m(S,\eta)
\leq{}\liminf_{T\to +\infty}
\frac{1}{T^{d-1}}m(T,\eta)
+\delta
+C\eta.
\end{equation*}
By the arbitrariness of $\delta$ we get \eqref{liminf_controlla_limsup}.
\end{proof}
\end{theorem}

\section{Approximation of nonlocal functionals}
\label{Approximation of non local functional}
In this section we show that if the decay condition  \eqref{eq:nuclei3} is not satisfied, then the $\Gamma$-limits of energies as in \eqref{rescale_energy} may be nonlocal. We provide two examples.\\

\begin{example}
    Fix $r\in(0,1)$ and assume $\Omega \subset \mathbb{R}^d$ such that $\diam(\Omega)>1-r$. Consider the family of functionals 
    \begin{equation*}
        \begin{split}
 G_\e(u)=G^{sr}_\varepsilon(u)+ G^{nl}(u):&= \frac{1}{\e^{d+1}}\int_{\{(x,y)\in \Omega \times \Omega:\, |y-x|\leq \e\}}|u(y)-u(x)|\, dx\, dy\\
&+\int_{\{(x,y)\in \Omega \times \Omega:\, 1-r\leq|y-x|\leq1+r\}}\left|u(y)-u(x)\right|\, dx\, dy.
        \end{split}
    \end{equation*}
 Note that, for $0<\e<1-r$, $G_\e$ is of the form \eqref{definizione di G} with $A=\Omega$, $\rho_\e(z)=\tilde\rho_\e(|z|)$ and $\tilde\rho_\e:[0,+\infty)\to[0,+\infty)$ defined by
    \begin{equation*}
    \tilde \rho_\e(t)=
        \begin{cases}
        \displaystyle
            \frac{t}{\e^{d+1}} &\text{if} \ 0\leq t\leq \e\\
            t&\text{if}\ 1-r\leq t\leq 1+r\\
            0 & \text{otherwise}.
        \end{cases}
    \end{equation*}
We find that $\rho_\e$ is uniformly bounded in $L^1(\rr^d)$, whereas \eqref{eq:nuclei3} fails. Indeed, by passing to polar coordinates,
    \begin{equation*}
\int_{\rr^d}\rho_\e(z) dz=\frac{\sigma_{d-1}}{d+1}\left[1+(1+r)^{d+1}-(1-r)^{d+1}\right],
    \end{equation*} 
while, given $\e<\eta\leq 1-r$, 
    \begin{equation*}
        \int_{|z|>\eta} \rho_\e(z)\, dz=\frac{\sigma_{d-1}}{d+1}\left[(1+r)^{d+1}-(1-r)^{d+1}\right]>0.
    \end{equation*}
    As shown in Example \ref{volume}, $ G^{sr}_\e$ $\Gamma$-converges with respect to the strong $L^1(\Omega;\rr^m)$ topology to the local functional 
\begin{equation}
\label{convergence_to_local_functional}
 G^{sr}(u):= \left\{
\begin{array}{ll}
\displaystyle \frac{2\omega_{d-1}}{d+1}\int_{S_u}|u^+-u^-|\,d\mathcal H^{d-1} &u\in BV(\Omega;Y)\\
+\infty &\hbox{otherwise.}
\end{array}
\right.
\end{equation}
We note that $G^{nl}$ is continuous with respect to the strong convergence in $L^1(\Omega;\rr^m)$. Indeed, through the change of variable $z=y-x$,  $G^{nl}(u)$ can be written as
    \begin{equation*}
G^{nl}(u)=\int_{B_{1+r}\setminus B_{1-r}}\int_{\{x \in \Omega\ :\ x+z\in \Omega\}}|u(x+z)-u(x)|\, dx\, dz.
    \end{equation*}
     Hence, given $u,v\in L^1(\Omega;\rr^m)$, we have
    \begin{equation*}
    \begin{split}
        |G^{nl}(v)-G^{nl}(u)|&\leq \int_{B_{1+r}\setminus B_{1-r}}\int_{\{x \in \Omega \ : \ x+z \in \Omega\}}|v(x+z)-u(x+z)|\, dx\, dz\\
        & + \int_{B_{1+r}\setminus B_{1-r}}\int_{\{x \in \Omega \ : \ x+z \in \Omega\}}|v(x)-u(x)|\, dx\, dz\\
        &\leq 2\, |B_{1+r}\setminus B_{1-r}|\|v-u\|_{L^1(\Omega;\rr^m)}.
    \end{split}
    \end{equation*}
     Then, the stability of $\Gamma$- convergence under continuous perturbations yields
    $$
    \Gamma(L^1)\text{-}\lim_{\e\to 0}G_\e(u)=G^{sr}(u)+G^{nl}(u).
    $$
   Hence, the first term is a local, perimeter-type energy, whereas the second is genuinely nonlocal.
\end{example}
    
\begin{example} Let $0<s<1$ and consider the family of functionals
\begin{equation*}
        \begin{split}
 \hat G_\e(u)=G^{sr}_\varepsilon(u)+ \hat G^{nl}(u):&=\frac{1}{\e^{d+1}}\int_{\{(x,y)\in \Omega \times \Omega:\, |y-x|\leq \e\}}|u(y)-u(x)|\, dx\, dy\\
&+\int_{\Omega}\int_{\Omega}\frac{|u(y)-u(x)|}{|y-x|^{d+s}}\, dx\, dy.
        \end{split}
\end{equation*}
    As in the previous example, $\hat G_\e$ is of the form \eqref{definizione di G} with $A=\Omega$,  $\rho_\e(z)=\hat\rho_\e(|z|)$ and $\hat\rho_\e:[0,+\infty)\to[0,+\infty)$ defined by
    \begin{equation*}
    \hat \rho_\e(t):=\chi_{[0,D]}(t)\frac{1}{t^{d+s-1}}+
        \begin{cases}
        \displaystyle
            \frac{t}{\e^{d+1}} &\text{if} \ 0\leq t\leq \e\\
            0&\text{if}\  t> \e,
        \end{cases}
    \end{equation*}
     where $D:=\diam(\Omega)$. Note that, if $Y=\{0,1\}$, $\hat G^{nl}(u)$ is twice the $s$-fractional perimeter of $\{u=1\}$ inside $\Omega$.
     Also in this case  $\rho_\e$ is uniformly bounded in $L^1(\rr^d)$, whereas \eqref{eq:nuclei3} fails. Indeed,
\begin{equation*}
    \int_{\rr^d}\hat\rho_\e(z)\, dz=\sigma_{d-1}\left[\frac{1}{d+1}+\frac{D^{1-s}}{1-s}\right],
\end{equation*}
while,
if $\e<\eta<D$, we have
\begin{equation*}
    \int_{|z|>\eta}\hat\rho_\e(z)\, dz=\frac{\sigma_{d-1}}{1-s}\left(D^{1-s}-\eta^{1-s}\right) >0.
\end{equation*}

We now prove  that $\hat G_\e\xrightarrow{\Gamma} G^{sr}+\hat G^{nl}$ in the strong $L^1(\Omega;\rr^m)$ topology, where $G^{sr}$ is defined in \eqref{convergence_to_local_functional}, but we cannot argue as in the previous example. Indeed, by Fatou's Lemma, we can easily infer  that $\hat G^{nl}(u)$ is lower semicontinuous with respect to strong  $L^1(\Omega;\rr^m)$ convergence, but it is not continuous. 

From the lower semicontinuity of $\hat G^{nl}$ we immediately deduce that the $\Gamma$-liminf inequality holds true. Indeed, for every $u_\e\to u$ in $L^1(\Omega;\rr^m)$, we have
\begin{equation*}
    \liminf_{\e\rightarrow 0} \hat G_\e(u_\e)\geq \liminf_{\e \rightarrow 0} G_\e^{sr}(u_\e)+\liminf_{\e \rightarrow 0}\hat G^{nl}(u_\e)\geq  G^{sr}(u)+\hat G^{nl}(u).
\end{equation*}
We now prove the $\Gamma$-limsup inequality. If $u\notin BV(\Omega;Y)$ there is nothing to prove,
since $G^{sr}(u)=+\infty$. If $u\in BV(\Omega;Y)$, we take $u_\varepsilon\equiv u$. By Theorem \ref{Thm_gamma_limite_esplicito},
applied with $A=\Omega$ and the kernel of Example \ref{volume}, we have
$\lim_{\varepsilon\to0}G^{sr}_\varepsilon(u)=G^{sr}(u)$. Hence, since $\hat G^{nl}$ does not depend on $\varepsilon$,
$$
\limsup_{\varepsilon\to0}\hat G_\varepsilon(u)=\lim_{\varepsilon\to0}G^{sr}_\varepsilon(u)+\hat G^{nl}(u)=G^{sr}(u)+\hat G^{nl}(u).
$$
\end{example}

\paragraph{\bf{Statement of AI use.}} The authors used AI tools for minor copyediting assistance during the final revision of this paper.
The manuscript and all mathematical content were conceived and written entirely by the authors.

\paragraph{\bf{Acknowledgements.}} The authors are members of Gruppo Nazionale per l'Analisi Ma\-te\-ma\-ti\-ca, la Probabilit\`a e le loro Applicazioni (GNAMPA) of INdAM. They all acknowledge support from the INdAM-GNAMPA Project ``Analisi e Gamma-convergenza per alcuni funzionali non locali" CUP E53C25002010001. The research of C.L. was also funded by the University of Naples Federico II through FRA Project ``ReSinApas". 
\\

\end{document}